\documentclass[11pt]{amsart}

\usepackage{etex}
\usepackage{latexsym}
\usepackage{amssymb}
\usepackage{amsmath}
\usepackage{amsfonts}
\usepackage[mathscr]{eucal}
\usepackage{MnSymbol}
\usepackage{color}
\usepackage{geometry}                		
\usepackage{graphicx}				
\usepackage{amssymb}
\usepackage{ upgreek }
\usepackage{multicol}              
\usepackage{stmaryrd}                  

\title{Brief Article}

\usepackage{mathtools}

\usepackage{a4wide}
\usepackage{amscd}

\usepackage{graphicx}
\usepackage{tikz}
	\usetikzlibrary{shapes,arrows,positioning,decorations.pathreplacing,calc}
	\usepackage{float}

\usepackage{amsthm}
\usepackage[all,cmtip]{xy}
\usepackage[pdfstartview=FitH]{hyperref}
\usepackage{pst-grad}
\usepackage{comment}
\usepackage{psfrag}
\usepackage{verbatim}
\usepackage{hyperref}
\hypersetup{colorlinks,allcolors=blue}
\usepackage{multirow}
\usepackage{caption}
\usepackage{tabularx}
\usepackage{caption}
\usepackage{soul}

\usepackage{lineno}

\newtheorem{thm}{Theorem}[section]
\newtheorem{cor}[thm]{Corollary}
\newtheorem{lem}[thm]{Lemma}
\newtheorem{prop}[thm]{Proposition}
\newtheorem{theorem}{Theorem}

\theoremstyle{definition}
\newtheorem{definition}[thm]{Definition}
\newtheorem{conj}[thm]{Conjecture}
\newtheorem{example}[thm]{Example}

\newtheorem{rem}[thm]{Remark}

\numberwithin{equation}{section}

\DeclareMathOperator{\Iso}{Isom}

\DeclareMathOperator{\curv}{curv}

\DeclareMathOperator{\rank}{rank}

\DeclareMathOperator{\im}{Im}
\DeclareMathOperator{\Kdim}{Krdim} 
\DeclareMathOperator{\depth}{depth}
\DeclareMathOperator{\trdeg}{trdeg}
\DeclareMathOperator{\Susp}{Susp}
\DeclareMathOperator{\diag}{diag}
\DeclareMathOperator{\grad}{grad}
\DeclareMathOperator{\Ann}{Ann}
\DeclareMathOperator{\Supp}{Supp}

\newcommand{\QQ}{\mathbb{Q}}

\newcommand{\SU}{\mathrm{SU}}
\newcommand{\Sp}{\mathrm{Sp}}

\newcommand{\ZZ}{\mathbb{Z}}

\newcommand{\B}{{\mathbf{B}}}                                      
\newcommand{\E}{{\mathbf{E}}}                                      

\renewcommand{\colon}{\nobreak\mskip2mu\mathpunct{}\nonscript
  \mkern-\thinmuskip{:}\mskip6muplus1mu\relax}


\begin{document}



\title[Equivariant cohomology of Alexandrov spaces]{On the equivariant cohomology of\\ cohomogeneity one  Alexandrov spaces}

\author[M.~ Amann]{Manuel Amann}
\author[M.~Zarei]{Masoumeh Zarei}

\date{\today}





\begin{abstract}
We give a characterization of those Alexandrov spaces admitting a cohomogeneity one action of a compact connected Lie group $G$ for which the action is Cohen--Macaulay. This generalizes a similar result for manifolds to the singular setting of Alexandrov spaces where, in contrast to the manifold case, we find several actions which are not Cohen--Macaulay.  In fact, we present results in a slightly more general context.

We extend the methods in this field by a conceptual approach on equivariant cohomology via rational homotopy theory using an explicit rational model for a double mapping cylinder.
\end{abstract}

\maketitle
\setcounter{tocdepth}{1}




\section{Introduction}

It is the goal of this article to bring together Riemannian and metric geometry (in the form of cohomogeneity one Alexandrov spaces) with the concept of Cohen--Macaulay modules 
from commutative algebra. The connection will be established by considering the equivariant cohomology algebras of the cohomogeneity one actions and by analyzing their algebraic features. For this we present some new approach to the field by transcribing equivariant cohomology to the realm of rational homotopy theory and by using algebraic models for cohomogeneity one spaces. This finally permits concrete computations. As an outcome we shall characterize those cohomogeneity one $G$-actions on Alexandrov spaces whose equivariant cohomology is Cohen--Macaulay.

\vspace{5mm}

Equivariant cohomology $H_G^*(X;\QQ)$ is an elaborate tool to study transformation groups on CW-complexes $X$. Recall that it can be defined as the cohomology of the Borel construction $X_G:=X\times_G \E G $, i.e.~$H_G^*(X;\QQ)=H^*(X\times_G \E G)$. It is of special interest when it happens to be particularly ``simple''. Various notions of such ``simplicity'' can be found in the literature, maybe most prominently featuring the term ``equivariant formality''. The content of the latter is that as an $H^*(\B G;\mathbb{Q})$-module the equivariant cohomology $H_G^*(X;\mathbb{Q})$ splits as a product of the cohomology of $X$ and the classifying space of $G$, i.e.~$H_G^*(X;\mathbb{Q})\cong H^*(X;\mathbb{Q})\otimes H^*(\B G;\mathbb{Q})$. Equivalently, the Leray--Serre spectral sequence of the Borel fibration $X\hookrightarrow{} X_G \to \B G$ degenerates at the $E_2$-term. This arises in several situations, including geometrically relevant contexts like simply-connected compact K\"ahler manifolds or Hamiltonian torus actions. Moreover, by Chang--Skjelbred and Atiyah--Bredon this property in the case of torus actions allows to reconstruct equivariant cohomology from lower dimensional orbit strata.

However, by standard localisation results an equivariantly formal torus action necessarily comes with fixed points, and clearly, therefore excludes free actions. There have been several attempts in the literature to provide variations to this notion. One, presented in \cite{GoRo}, is the notion of a \emph{Cohen--Macaulay $G$-action}. Recall that a module over a ring is said to be Cohen--Macaulay if its dimension equals its depth (see Section~\ref{S:CM} for further details). This is a prominent and abundant concept in Algebraic Geometry which, for example, can be used to suitably generalize regular schemes, etc. In the graded context applied to the equivariant cohomology algebra $H_G^*(X;\mathbb{Q})$ considered as an $H^*(\B G;\mathbb{Q})$-module, it yields the definition of a Cohen--Macaulay $G$-action. Note that, in particular, this does comprise free $G$-actions on the one hand, and is readily implied by equivariant formality on the other hand.

In \cite{GoMa14, GoMa18} it was shown that homogeneous spaces and cohomogeneity one manifolds are Cohen--Macaulay. For this recall that a manifold is of cohomogeneity one (obviously generalising transitive $G$-actions) if it permits a smooth $G$-action with an orbit of codimension one.

We recall that due to the existence of a biinvariant metric on a compact Lie group together with a theorem by O'Neill homogeneous spaces are the prime examples of manifolds with non-negative sectional curvature. Moreover, under certain restrictions such non-negatively curved metrics were also found on large classes of cohomogeneity one manifolds, and these spaces recently have led to the discovery of new positively curved examples.

We remark that there has always been an intriguing and deep interplay between Riemannian manifolds with such lower curvature bounds and the existence of symmetries upon them. On many known examples isometry groups are rather large. Hence it seems reasonable to speculate about the topological nature of such actions of compact Lie group actions. Let us propose one conjecture in this direction.
\begin{conj}\label{conj}
Suppose a compact Lie group $G$ acts isometrically on the simply-connected manifold $(M,g)$ of positive sectional curvature. Then the action is equivariantly formal.
\end{conj}
As one motivation for the conjecture recall that the action is equivariantly formal if and only if so is the induced action by the maximal torus $T\subseteq G$. Note further that if $M$ is even-dimensional, the Hopf conjecture speculates that $\chi(M)>0$. A confirmation of the Bott--Grove--Halperin conjecture would yield that $M$ is rationally elliptic, whence, both taken together, would imply that $H^*(M;\QQ)=H^\textrm{even}(M;\QQ)$ by the structure theory of such positively elliptic spaces. Since also $H^*(\B G, \QQ)$ is concentrated in even degrees, the spectral sequence of the Borel fibration then would degenerate at the $E_2$-term for lacunary reasons. Note further that the ``fixed-point-obstruction'' to equivariant formality in positive curvature is removed by the Weinstein fixed-point theorem applied to a topological generator of $T$.

We recall that a Cohen--Macaulay action with a fixed-point is known to be equivariantly formal whence the concepts agree for even dimensional positively curved manifolds by the Weinstein fixed-point theorem.

\vspace{5mm}

Alexandrov spaces are metric spaces generalising manifolds with a lower curvature bound basically using the purely metric Toponogov characterization of sectional curvature bounds as a definition. They are of particular importance as they close the category of manifolds with lower curvature bounds under Gromov--Hausdorff convergence or quotients of compact Lie group actions. Both their geometry and topology have undergone intense studies. Investigating their equivariant cohomology, however, still seems to be a rather new and interesting field. In particular, we may consider cohomogeneity one Alexandrov spaces defined in complete analogy to the manifold setting.

Several properties from non-negatively/positively curved manifolds can be transferred to Alexandrov spaces. Hence it is natural to think about the above conjecture within the setting of positively curved Alexandrov spaces. Here, however, one finds an immediate easy counterexample provided by the standard cohomogeneity one action of $\mathbf{SU}(3)$ on the spherical suspension $\Susp(W^7_{1, 1})$ of $W^7_{1, 1}$. It is easy to see that this action is actually not even Cohen--Macaulay (see Page \pageref{PageEx}).

We can motivate Conjecture \ref{conj} further by actually showing that in order to prove equivariant formality we only need to prove the Cohen--Macaulay property.
\begin{rem}
If $X$ is a positively curved cohomogeneity one Alexandrov space, then $X$ is Cohen--Macaulay if and only if it is equivariantly formal provided that $\chi(X)\neq 0$ in the case when $\dim X$ is odd---see Proposition \ref{prop:CM_ef}.
\end{rem}

This further provides good motivation to analyze which cohomogeneity one Alexandrov spaces actually are Cohen--Macaulay.  Here, let us point out that we can actually do this for a larger class of spaces which contains cohomogeneity one Alexandrov spaces. Before stating the theorem (which then can directly be applied to Alexandrov spaces) in this larger context  let us establish this more general framework of \emph{generalized cohomogeneity one spaces}.


In order to define such a ``generalized cohomogeneity one space'' let $(G, H, K^-, K^+)$ be a quadruple of compact Lie groups such that $G$ is connected and $H\subseteq K^{\pm}\subseteq G$. Let 

\begin{align}\label{Eq:decomposition_}
X = G\times_{K^-}C(K^-/H)\bigcup_{G/H} G\times_{K^+}  C(K^+/H),
\end{align}
where $C(K^{\pm}/H)$ is a closed cone over $K^{\pm}/H$. Then there is an action of $G$ on $X$ with  three orbit types $G/K^-$, $G/H$ and $G/K^+$, and such that the  orbit space is a  closed interval. The orbits of types $G/K^-$ and $G/K^+$ are mapped to the boundary of the orbit space  under the natural projection map, and the orbits of type $G/H$ are mapped to the interior points of the orbit space. We call the $G$-space $X$ a \emph{generalized cohomogeneity one space with group diagram $(G, H, K^-, K^+)$}. 

Note that due to the known double mapping cylinder decomposition of closed simply-connected cohomogeneity one Alexandrov spaces, we obtain Splitting \eqref{Eq:decomposition_} with $G$ the group acting by cohomogeneity one, $K^\pm$ the singular isotropy groups, and $H$ the principal isotropy group.

In this situation the fact that $K^\pm/H$ are positively curved homogeneous spaces--- in view of the classification of the latter (see \cite{WZ})---yields $\rank K^\pm -\rank H\leq 1$. It follows that the next theorem is directly applicable to cohomogeneity one Alexandrov spaces $X$.


\begin{theorem}\label{T:Main_THM_A}
Let $X$ be a generalized cohomogeneity one space with group diagram \linebreak[4]$(G, H, K^-, K^+)$, where $K^{\pm}/H$ are connected homogeneous spaces of positive dimension and    $$\max \{\rank K^-, \rank K^+\}\leq \rank H+1.$$  Moreover, suppose that the classifying spaces of the isotropy groups  $H$, $K^-$, and $K^+$ are Sullivan  spaces.
Then $H^*_G(X, \QQ)$ is a Cohen--Macaulay $H^*(\B G, \QQ)$-module  if and only if one of
the following statements holds.
\begin{enumerate}
\item  $\rank H=\rank K^-=\rank K^+$.
\item  $\rank H< \max \{\rank K^-, \rank K^+\}$ and $\im H^*(\B\iota^{-})+\im H^*(\B\iota^{+})=H^{*}(\B H, \QQ)$.
\end{enumerate}
\end{theorem}


\bigskip

For the definition of ``Sullivan spaces'' see Definition \ref{D:Sullivan_Space}. Note that the classifying spaces are automatically Sullivan if the respective Lie groups are connected.

\begin{rem}
The result is sharp according to the following two ways of looking at it: 
\begin{itemize}
\item
First, the condition on rank is necessary, as the example of the generalized cohomogeneity one Alexandrov space given by the group diagram $(T^n, 1,S^1,T^n)$  with $\rank K^+ -\rank H=n\geq 2$ shows (see Example \ref{E:Corank}).
%
%
\item Second, the condition $\im H^*(\B\iota^{-})+\im H^*(\B\iota^{+})=H^{*}(\B H, \QQ)$ can occur even if none of the morphisms $H^*(\B\iota^{\pm})$ is surjective as an example with group diagram $(Sp(1)\times Sp(1)\times Sp(1), S^1\times S^1, S^1\times Sp(1)\times Sp(1), Sp(1) \times S^1\times Sp(1))$ shows (see Example \ref{E:Surjectivity}).
\end{itemize}
\end{rem}

Already in the category of Alexandrov spaces we can realize the failure of the Cohen--Macaulay property for most normal fibres $K^\pm/H$:
\begin{theorem}\label{T:Main_THM_B}
Choose  positively curved homogeneous spaces $F^+$ and $ F^-$  such that $F^\pm$ can be written as a quotient of compact Lie groups whose classifying  spaces are Sullivan spaces. Then there exists a cohomogeneity one Alexandrov space with group diagram $(G, H, K^{-}, K^{+})$---such that the classifying spaces $\B H$, $\B K^\pm$ are  Sullivan spaces with  $F^\pm\approx K^{\pm}/H$---which is \emph{not} Cohen--Macaulay
 if and only if
\begin{itemize}
\item
$\rank H<\max\{\rank K^{-}, \rank K^{+}\}$, and
\item
$K^{\pm}/H\in \{W^{7}_{p, q}/\Gamma, B^{13}, M^\textrm{even}\}$,
\end{itemize}
where, $M^\textrm{even}$ is an even-dimensional positively curved homogeneous space, and $W^7_{p, q}/\Gamma$ is a positively curved homogeneous space whose universal covering  space is the   Aloff--Wallach space  $W^7_{p, q}$ and $B^{13}$ is the $13$-dimensional Berger space.
\end{theorem}

\begin{rem}
We remark that in order to prove this result we do extend known models for homogeneous spaces $G/H$ of compact Lie groups to the case that $G$ is connected and $\B H$ is a Sullivan space (yet $H$ not necessarily connected). In particular, we prove that odd-dimensional positively curved homogeneous spaces are rationally nilpotent (see Proposition \ref{P:non-simply-connected_homogeneous_space_Rationally_Nilpotent}).
\end{rem}

From the latter observation on odd-dimensional positively curved homogeneous spaces one directly derives that they are all Sullivan spaces, in particular. We do not know, however, if this implies that in Theorems \ref{T:Main_THM_A}  and \ref{T:Main_THM_B} we can drop the condition that $\B H$, $\B K^-$, and $\B K^+$ be Sullivan  spaces.

As a consequence of Theorem \ref{T:Main_THM_A} we can generalize the positive result that cohomogeneity one actions on manifolds are Cohen--Macaulay to orbifolds.

\begin{theorem}\label{T:Cohom_1_Orbifolds}
Let $X$ be a closed simply-connected  smooth orbifold and let $G$ be a compact connected Lie group
which acts on   $X$ by cohomogeneity one with  a group diagram $(G, H, K^-, K^+)$, where the classifying spaces of the isotropy groups  $H$, $K^-$, and $K^+$ are Sullivan  spaces. Then $H^*_G(X, \QQ)$ is a Cohen--Macaulay $H^*(\B G)$-module.
\end{theorem}

This allows us to provide the following schematic diagram of cohomogeneity one spaces. On the upper half we specify the cohomogeneity one space, on the lower half the fibres $K^\pm/H$. For the inner two shells the cohomogeneity one action is known to be Cohen--Macaulay due to \cite{GoMa14} and \cite{GoMa18}. Under mild technical assumptions we hence provide this for orbifolds, and first examples of non-Cohen--Macaulay actions can be found in the outer shell of Alexandrov spaces.
\begin{center}
\includegraphics[width=1.0\textwidth]{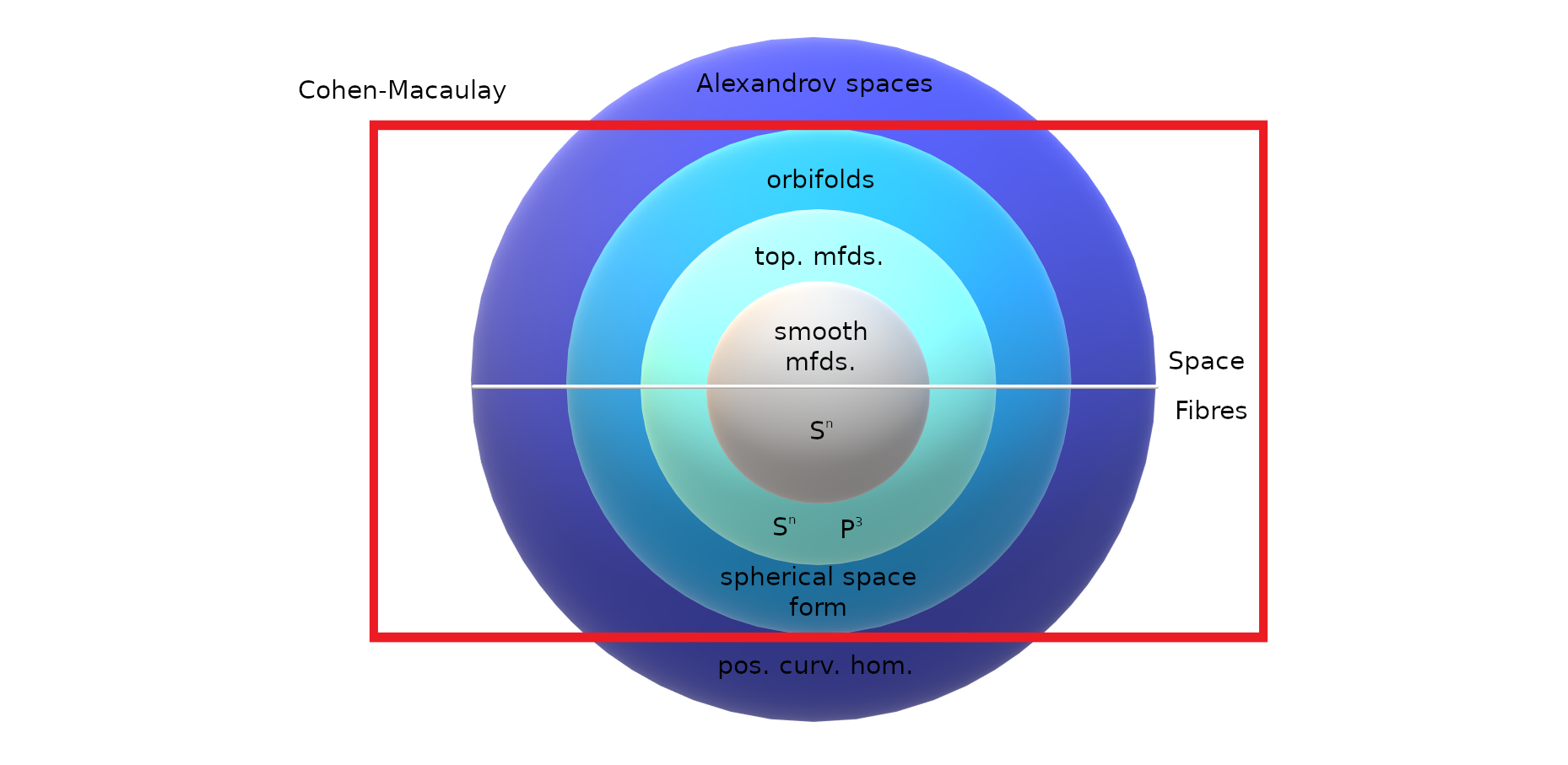}
\end{center}

Note that in the Riemannian setting cohomogeneity one actions are special cases of \emph{hyperpolar} actions. In \cite{GoHaghMa} the authors generalize the results to hyperpolar actions on symmetric spaces of compact type. More generally, in \cite{CaGoHeMa} the authors provide a formula to compute
the  equivariant cohomology ring of a double mapping cylinder. Then they apply the formula to compute the equivariant cohomology ring of a cohomogeneity one action of a compact Lie group $G$ on a manifold $M$ with a group diagram $(G, H, K^{-}, K^{+})$ imposing an orientability  condition on the fibrations $BK^{\pm}\to BH$.

\vspace{5mm}

Let us quickly sketch the principal ideas for proving our main  Theorems. 
 In analogy to the manifold case, a cohomogeneity one Alexandrov space  with group diagram $(G,H, K^-,K^+)$ admits a decomposition as a double mapping cylinder, i.e.~as two mapping cylinders over the singular orbits which we glue at the common principal orbit $G/H$.   The double mapping cylinder decomposition gives us the platform to carry out the proof. It is worth mentioning that for the  group diagram $(G,H, K^-,K^+)$ the spaces  $K^-/H$ and $K^+/H$ are positively curved homogeneous Alexandrov spaces and hence positively curved homogeneous spaces (see Theorem \ref{T:ALEX_STRUCTURE}). For their classification upon which our proof builds see \cite{WZ} . In particular, this classification yields that $\dim K^\pm/H \mod 2= \rank K^\pm-\rank H$. This places the Alexandrov spaces in the depicted setting of generalized cohomogeneity one spaces with rank restrictions in Theorem \ref{T:Main_THM_A}.

The upshot of the proof is to identify when the induced morphism \linebreak[4]$H^*(\B K^ \pm,\QQ) \to H^*(\B H,\QQ)$ is injective respectively surjective, and, in particular, to connect this, respectively the property when their sum is surjective,  to being Cohen--Macaulay.
On the one hand, we draw on the known techniques,  which we adapt to the setting of generalized cohomogeneity one respectively Alexandrov spaces, in order to present sufficient conditions for equivariant cohomology to be Cohen--Macaulay. On the other hand, we draw new ideas and techniques from rational homotopy theory which provide us with a concrete understanding of the Cohen--Macaulay property. This then allows us to give tailored arguments by which we can decide whether this property holds or not.

More precisely, the double mapping cylinder decomposition of $X$ gives a rational model for equivariant cohomology. It enables us to get control on the structure of equivariant cohomology, in particular, to rule out certain homogeneous spaces as candidates for the normal fibers of a generalized cohomogeneity one space. We point out that it is one neat feature of the proof that it is then indeed possible to see explicitly how regular sequences and prime ideals are related and interact in these situations.

As for the proof of Theorem \ref{T:Main_THM_B} we additionally draw on the join construction for positively curved Alexandrov spaces in order to construct the depicted examples.

We hope that this new explicit description of equivariant cohomology will come in handy for many further problems in this area and may constitute an additional helpful toolset.

\vspace{3mm}

\noindent\textbf{Structure of the article.}
In Section \ref{S:Alexandrov_spaces}  we recall the definition of generalized cohomogeneity one spaces and of Alexandrov spaces. We further review some basic facts about cohomogeneity one Alexandrov spaces. In Section \ref{S:Sullivan_spaces} we collect the relevant information about rationally nilpotent and Sullivan spaces needed for our arguments. The algebraic model for the equivariant cohomology of a generalized cohomogeneity one space, which is the cornerstone of the proof of Theorem~\ref{T:Main_THM_A}, is presented in Section~\ref{S_ Sullivan_model_for_double_mapping_cylinder}. In Section~\ref{S:CM}  we recall the definitions and basic facts of Cohen--Macaulay modules. Section~\ref{S:Proof} is  devoted to the proof of the main theorems. In Section~\ref{S:PC} we turn our attention to positively curved cohomogeneity one Alexandrov spaces and show that when the Euler characteristic of the space in nonzero, equivariant formality and the Cohen--Macaulay property agree.

\vspace{3mm}

\noindent\textbf{Acknowledgements.}
The authors want to express their gratitude to Steve Halperin for helpful discussions on rational nilpotence.

The first named author was supported both by a Heisenberg grant and his research grant AM 342/4-1 of the German Research Foundation. The second named author was also funded by DFG research grant AM 342/4-1. Both authors are 
associated to the DFG Priority Programme 2026, ``Geometry at Infinity''.

\section{Cohomogeneity one  actions
}
\label{S:Alexandrov_spaces}

In this section we first recall some basics about cohomogeneity one actions and, in a second step, focus on Alexandrov spaces equipped with such actions.


We first recall the definition of a generalized cohomogeneity one action. Let $(G, H, K^-, K^+)$ be a quadruple of compact Lie groups such that $G$ is connected and $H\subseteq K^{\pm}\subseteq G$. Let 

\begin{align}\label{Eq:decomposition_2}
X = G\times_{K^-}C(K^-/H)\bigcup_{G/H} G\times_{K^+}  C(K^+/H),
\end{align}
where $C(K^{\pm}/H)$ is a closed cone over the respective homogeneous space $K^{\pm}/H$. Then there is an action of $G$ on $X$ with  three orbit types $G/K^-$, $G/H$ and $G/K^+$, and such that the  orbit space is a  closed interval. The orbits of types $G/K^-$ and $G/K^+$ are mapped to the boundary of the orbit space  under the natural projection map, and the orbits of type $G/H$ are mapped to the interior points of the orbit space. We call the $G$-space $X$ a \emph{generalized cohomogeneity one space with group diagram $(G, H, K^-, K^+)$}.
Note that by \eqref{Eq:decomposition_2}, the space $X$ can be written as a double mapping cylinder of the following maps 

\begin{align*}
K^{\pm}/H\to G/H\to G/K^{\pm}.
\end{align*}


\begin{definition}
A length space $(X, d)$ of finite (Hausdorff) dimension has \emph{curvature bounded from below} by $k$, denoted by $\curv(X)\geq k$, if every point $x\in X$ has a (sufficiently small) neighborhood $U$ such that, for any collection
of four different points $(x_0, x_1, x_2, x_3)$ in $U$, the following condition holds:
\begin{equation*}
\tilde{\angle}_k{x_1x_0x_2} + \tilde{\angle} _k{x_2x_0x_3} + \tilde{\angle} _k{x_3x_0x_1}  \leq 2\pi.
\end{equation*}
Here, $\tilde{\angle} _k{x_ix_0x_j} $, called the \textit{comparison angle}, is the angle at $\tilde{x}_0$ in the geodesic triangle in $M^2_k$, the simply-connected Riemannian $2$-manifold with constant curvature $k$, with  vertices \linebreak $(\tilde{x}_0, \tilde{x}_i, \tilde{x}_j )$ such that  $d(x_0, x_i)=d(\tilde{x}_0, \tilde{x}_i)$, $d(x_0, x_j)=d(\tilde{x}_0, \tilde{x}_j)$ and $d(x_j, x_i)=d(\tilde{x}_j, \tilde{x}_i)$.
An \emph{Alexandrov space} is a complete length space with finite (Hausdorff) dimension and curvature bounded from below by $k$ for some $k\in \mathbb{R}$.
\end{definition}
Note that in this case Hausdorff dimension  is actually a non-negative integer.

If $(M, g)$ is a complete Riemannian manifold with sectional curvature bounded from below by $k$, then it follows from Toponogov's Theorem that $(M, d)$ is an Alexandrov space with $\curv((M,g))\geq k$, where $d$ is the distance function on $M$ induced by $g$. Therefore, Alexandrov spaces are a synthetic generalization of complete Riemannian manifolds with lower sectional curvature bounds.

The \emph{space of directions} of a general Alexandrov space $X^n$ of dimension $n$ at a point $x$ is, by definition, the completion of the space of \emph{geodesic directions} at $x$. Recall that two geodesics emanating from $x$ define the same geodesic direction if the angle between them is zero (see \cite[Page~100]{Burago}).
We will denote it by $\Sigma_xX^n$. It is a compact Alexandrov space of dimension $n-1$ with curvature bounded from below by $1$.

For an $n$-dimensional Alexandrov space $X$, Fukaya and Yamaguchi proved in \cite[Theorem 1.1]{FY} that $\Iso(X)$, the isometry group of $X$, is a Lie group. Moreover, if $X$ is compact and connected, then $\Iso(X)$ is compact  (see \cite[Page~370, Satz I]{DW}).
As in the Riemannian case, the maximal dimension of $\Iso(X)$ is $n(n+1)/2$ and, if equality holds, $X$ must be isometric to a Riemannian manifold (see \cite[Theorems 3.1 and 4.1]{GGG}).


In analogy to locally smooth actions (see \cite[Ch.~IV, Section 3]{Bredon}) for an isometric action of a compact Lie group $G$ on an Alexandrov space $X$ there also exists a maximal orbit type $G/H$  (see \cite[Theorem 2.2]{GGG}). This orbit type is the \emph{principal orbit type} and orbits of this type are called \emph{principal orbits}. A non-principal orbit is \emph{exceptional} if it has the same dimension as a principal orbit. If it has strictly lower dimension, it is called \emph{singular}.


\vspace{5mm}

Now we collect some basic facts on cohomogeneity one Alexandrov spaces. For more details we refer the reader to \cite{GS} or \cite{GGZAlex}.


\begin{definition}
Let $G$ be a compact connected Lie group which acts isometrically on an Alexandrov space  $X$. Let $G(x)$ be an orbit of $X$. We define the \emph{normal space of directions} to $G(x)$, denoted by $S_{x}^{\perp}$ as follows
\begin{align*}
S_{x}^{\perp} =\{ v \in \Sigma_x X \mid d(v,w) = \pi/2\,\,\,  \text{for all } w\in S_{x}\},
\end{align*}
where $S_{x}$ is the unit tangent space to the orbit $G(x)$.
\end{definition}

Now we turn our attention to cohomogeneity one Alexandrov spaces and recall their structure. 

Recall that the orbit space $X/G$ of an Alexandrov space $X$ by an isometric action of a group $G$ with closed orbits is again an Alexandrov space (see \cite[Proposition 10.2.4]{Burago}).
\begin{definition}
Let $X$ be a connected $n$-dimensional Alexandrov space with an isometric action of a compact connected Lie group  $G$. The action is of \emph{cohomogeneity one} if the orbit space is one-dimensional or, equivalently, if the principal orbit is of dimension $n-1$. We call a connected Alexandrov space with an isometric action of cohomogeneity one a \emph{cohomogeneity one Alexandrov space}.
\end{definition}

Since one-dimensional Alexandrov spaces are topological manifolds, the orbit space of a cohomogeneity one Alexandrov space is homeomorphic to a connected $1$-manifold (possibly with boundary). If $X$ is a closed Alexandrov space, i.e.~compact without boundary, it  must be either a circle or a closed interval. When the orbit space is homeomorphic to $[-1,1]$, there are three types of isotropy groups:
by the Isotropy Lemma (see \cite[Lemma 2.1]{GGG}) and the fact that principal orbits are open and dense the orbits corresponding to the interior of $[-1,1]$ are all of the form $G/H$ up to conjugation of $H$, the \emph{principal isotropy group}. The non-principal orbits corresponding to $\pm 1$ are of the form $G/K^{\pm}$ with exceptional respectively singular isotropy groups $K^{\pm}$. It follows that $K^-\supseteq H\subseteq K^+$ and the \emph{group diagram} $(G,H,K^-,K^+)$ formed by groups and inclusions then---whenever it can be realized as a cohomogeneity one Alexandrov space---uniquely determines its homeomorphism type. (See Theorem \ref{T:ALEX_STRUCTURE} for a characterization of when it can be realized). Note further that $(G,H,K^-,K^+)$ and $(G,H,K^+,K^-)$ are $G$-equivariantly homeomorphic.

Note further that $H$ is indeed a proper subgroup of $K^\pm$. For this it suffices to observe that if $\dim K^\pm=\dim H$, then $K^\pm\neq H$. Actually, in this case, the normal space of directions $S^\perp$ satisfies $S^\perp = S^0$ with a transitive action of $K^\pm$ with isotropy $H$. Hence $K^\pm/H = S^0$, which shows that $K^\pm \neq H$.

The following theorem determines the structure of closed cohomogeneity one Alexandrov spaces with orbit space an interval.


\begin{thm}[\protect{\cite[Theorem A]{GS}}]
\label{T:ALEX_STRUCTURE}
Let $X$ be a closed Alexandrov space with an effective isometric $G$-action of cohomogeneity one with principal isotropy $H$ and orbit space homeomorphic to $[-1,1]$. Then $X$ is the union of two
fiber bundles over the two singular orbits whose fibers are closed cones over positively curved homogeneous  spaces, that is,
\begin{align}\label{Eq:decomposition}
X = G\times_{K^-}C(K^-/H)\bigcup_{G/H} G\times_{K^+}  C(K^+/H).
\end{align}
 The group diagram of the action is given by $(G, H,  K^-, K^+)$, where $K^\pm/H$ are positively curved homogeneous spaces.
 Conversely, a group diagram $(G, H,  K^-, K^+)$, where $K^\pm/H$ are positively curved homogeneous spaces, determines a cohomogeneity one Alexandrov space (uniquely determined up to homeomorphism).
\end{thm}
We remark that this can equivalently be phrased using that a map $f\colon X\to Y$  between two topological spaces up to homotopy can be deformed to the inclusion $X\hookrightarrow{} M_f\simeq Y$ into the mapping cylinder at time $0$. The gluing above then corresponds to gluing the respective two mapping cylinders of the bundles $G/H \to G/K^\pm$ at time $1$ in order to obtain the so-called \emph{double mapping cylinder} (see \cite{GroveHalperin}).

We further remark that  we may refer to $K^{\pm}/H$ as the \emph{singular normal fibers}.

As Theorem \ref{T:ALEX_STRUCTURE} demonstrates,  cohomogeneity one Alexandrov spaces with group diagrams $(G, H,  K^-, K^+)$ are indeed  special instances of generalized cohomogeneity one spaces. 

\begin{rem}
As we stated already in the introduction, the theory of closed cohomogeneity one Alexandrov spaces extends the one of closed cohomogeneity one manifolds.
Indeed,  recall that the structure theorem \ref{T:ALEX_STRUCTURE} we presented for Alexandrov spaces has its original counterpart in the category of smooth manifolds (where group diagrams actually determine cohomogeneity one manifolds up to diffeomorphism). That is, given a group diagram with normal fibers smooth spheres we may realize it as a cohomogeneity one manifold. Furthermore, by choosing standard metrics we do this in such a way that the normal spheres are round whence they are positively curved. Thus, up to $G$-equivariant homeomorphism, such a closed cohomogeneity one manifold is a cohomogeneity one Alexandrov space by Decomposition \eqref{Eq:decomposition}.

The analog holds true when replacing manifolds by orbifolds (see \cite{Gonzalez}), i.e.~singular normal fibers are no longer necessarily smooth spheres but spherical space forms.
\end{rem}

\vspace{5mm}

Positively curved homogeneous spaces originally have been classified by Wallach and B\'erard--Bergery (see \cite{WZ} for a modern, self-contained, complete proof). In the simply-connected case these are classified to come out of the list of compact rank one symmetric spaces, flag manifolds $W^6$, $W^{12}$, $W^{24}$, Aloff--Wallach spaces $W^{7}_{p,q}$ and Berger spaces $B^7$ and $B^{13}$. Whilst $B^7$ is rationally a sphere, note that $W^7_{p,q}$ has the rational type of $S^2\times S^5$ and $B^{13}$ is rationally equivalent to $\mathbb{CP}^2\times S^9$ (see Section \label{S:Sullivan_spaces} for the definition of rational type).

We conclude this section by  special examples of  cohomogeneity one Alexandrov spaces that we need later in the article.


\begin{definition}[Suspension action]
Let $(X,d)$ be an Alexandrov space with curvature bounded from below by $1$. Define the \emph{spherical suspension} over $(X,d)$ by $\Susp(X):= \frac{X\times [0,\pi]}{\sim}$ (where $\sim$ contracts $X\times \{0\}$ and $X\times \{\pi\}$ to a point respectively) together with the positively curved suspension metric.

Let $G$ be a Lie group which acts isometrically on an  Alexandrov space $X$.  
We call the action of $G$ on  $\Susp(X)$ defined by 
\[
g\cdot [(x, t)]=[(gx, t)].
\]
the \textit{suspension action}.
\end{definition}


\begin{prop}\cite[Proposition~2.27]{GGZAlex}\label{P:Suspension-action}
Let  $G$ act transitively on a positively curved homogeneous space $M$ with isotropy group $H$.  Then the suspension action of  $G$ on  $\Susp(M)$ is of cohomogeneity one with diagram
$(G,  H,  G,  G)$. Conversely, a cohomogeneity one action of  $G$ with the above group diagram, and $G/H$  a positively curved  homogeneous space,  is  equivariantly homeomorphic to the  suspension action of  $G$ on $\Susp(G/H)$.
\end{prop}

\begin{definition}
Let $(X,d_X)$, $(Y,d_Y)$ be two Alexandrov spaces with curvature bounded from below by $1$.
The \emph{(topological) join} of  $X$ and $Y$ is the space
\[
X\ast Y = (X\times Y\times [0, \pi/2])/\sim,
\]
where $(x_1, y_1, t_1)\sim (x_2, y_2,  t_2)$, if and only if $t_1=t_2=0$ and $x_1=x_2$ or $t_1=t_2=\pi/2$ and $y_1=y_2$. We endow $X\ast Y$ with a metric defined by
\[
\cos (d([x_1, y_1, t_1], [x_2, y_2,  t_2]))=\cos t_1\cos t_2\cos d_{X}(x_1, x_2)+\sin t_1\sin t_2\cos d_{Y}(y_1, y_2).
\]
The space $(X\ast Y,d)$ is the \emph{spherical join} of $(X,d_X)$ and $(X,d_Y)$ and  is an Alexandrov space with $\curv\geq 1$.
\end{definition}


Let $G_1$ and $G_2$ be two Lie groups which act on  Alexandrov spaces $X_1$ and $X_2$, respectively. The  action of $G_1\times G_2$ on  $X_1\ast X_2$ is called \textit{join action}, if $G_1\times G_2$ acts on  $X_1\ast X_2$ naturally, i.e.
\[
(g_1, g_2)\cdot [(x, y, t)]=[(g_1x, g_2y, t)].
\]


\begin{prop}\cite[Proposition~2.28]{GGZAlex}\label{P:Join_Action}
If two Lie groups $G_1$ and $G_2$ act transitively on  positively curved homogeneous spaces $M_1$ and $M_2$ with isotropy groups $H_1$ and $H_2$, respectively, then the join action of  $G=G_1\times G_2$ on  $M_1\ast M_2$ is of cohomogeneity one with the following diagram:
\[
(G_1\times G_2,  H_1\times H_2,  G_1\times H_2,  H_1\times G_2).
\]
Conversely, a cohomogeneity one action of  $G_1\times G_2$ with the above group diagram, and $G_i/H_i$  positively curved  homogeneous spaces, for $i=1,2$, is  equivalent  to the  join action of  $G$ on $(G_1/H_1)\ast (G_2/H_2)$.
\end{prop}

\section{Rationally nilpotent and Sullivan spaces}\label{S:Sullivan_spaces}

This section cannot provide an introduction to rational homotopy theory. For an elaborate discussion of the latter we refer the reader to \cite{RHT} and \cite{RHT_II}.

Here, let us merely recall some concepts relevant for this article.

The main objects in the version of rational homotopy theory elaborated by Sullivan are commutative differential graded algebras $(A,d)$ together with their morphisms respectively cochain algebras, i.e.~those which are concentrated in non--negative degrees. One special class of such morphisms are so-called \emph{quasi--isomorphisms}, i.e.~those which induce isomorphisms on cohomology. We call two commutative differential graded algebras $(A,d)$ and $(B,d)$ \emph{weakly equivalent} if there is a chain, a zig-zag, of quasi-isomorphisms
\begin{align*}
(A,d) \xrightarrow{\simeq} \ldots \xleftarrow{\simeq} \ldots \xleftarrow{\simeq} (B,d).
\end{align*}

\begin{definition}\cite[Section~12, Page 138]{RHT}
A \emph{Sullivan algebra} is a commutative cochain algebra of the form $(\Lambda V, d)$ where
\begin{itemize}
\item $ V = \{V^{p}\}_{p\geq 1} $, and $\Lambda V$ denotes the free graded commutative algebra on $V$;
\item  $V = \bigcup_{k=0}^{\infty} V(k)$, where $V(0)\subset V(1)\subset \ldots$ is an increasing sequence of
graded subspaces such that
$$d=0 \quad \text{in} \quad V(0) \quad  \text{and} \quad d\colon V(k)\to \Lambda V(k-1), \quad k\geq 1.$$
 \end{itemize}
\end{definition}

\begin{definition}\cite[Section~12, Page 138]{RHT}
\begin{itemize}
\item
A \emph{Sullivan model} for a commutative cochain algebra $(A, d) $ is a quasi--isomorphism
\[
m \colon (\Lambda V, d)  \to  (A,d)
\]
from a Sullivan algebra $(\Lambda V, d)$.
\item
A Sullivan algebra  $(\Lambda V, d)$ is called \emph{minimal} if
$$\im(d)\subseteq \Lambda ^{+} V \cdot \Lambda ^{+} V. $$
\end{itemize}
\end{definition}

 We now relate topological spaces to the theory of commutative cochain algebras. This transmission is carried out via \emph{Sullivan's functor} $A_{PL}$, which is defined based on a simplicial construction mimicking the algebra of differential forms on a smooth manifold. Throughout the article we freely  use this functor and refer the reader to  \cite[Section~10]{RHT} for more details.

 Recall that we refer to two (not necessarily nilpotent) spaces $X$, $Y$ as being \emph{rationally equivalent} or \emph{of the same rational homotopy type},  if the differential graded algebras of polynomial diffferential forms $\operatorname{A}_{\textrm{PL}}(X)$ and $\operatorname{A}_{\textrm{PL}}(Y)$ are weakly equivalent. (In order to avoid confusion, let us state here already that we shall only apply this definition to \emph{Sullivan spaces} such that it does become a reasonable definition in the sense that it preserves rational homotopy groups for example.)

\begin{definition}\cite[Section~12, Page 138]{RHT}
 If $X$ is a path connected topological space, then  a Sullivan model for $(A_{PL}(X), d)$
is called a Sullivan model for $X$.
\end{definition}

\begin{definition}\cite[Section~14, Page 181]{RHT}
A \emph{relative Sullivan algebra} is a commutative  cochain algebra of the form $(B\otimes \Lambda V,  d) $ where
\begin{itemize}
\item $(B, d)=(B\otimes 1, d)$ is a subcochain algebra and $H^{0}(B)=\QQ$;
\item $1\otimes V=V=\{V^{p}\}_{p\geq 1}$;
\item $V=\bigcup_{k=0}^{\infty}V(k)$, where $V(0)\subseteq V(1)\subseteq \ldots $ is and increasing sequence of graded subspaces  such that
$$d \colon V(0)\to B \quad  \text{and} \quad d \colon V(k)\to B\otimes \Lambda V(k-1), \quad k\geq 1.$$
 \end{itemize}
\end{definition}

Let $\varphi \colon (B, d)\to (C, d)$ be a morphism of commutative cochain algebras with $H^{0}(B)=\QQ$.  A Sullivan model for $\varphi$ is defined as follows:

\begin{definition}\cite[Section~14, Page 181]{RHT}
A Sullivan model for $\varphi$ is a quasi-isomorphism of cochain algebras
\[
m \colon (B\otimes \Lambda V,  d) \to  (C,d)
\]
such that $(B\otimes \Lambda V,  d) $ is a relative Sullivan algebra with base $(B, d)$ and $m\mid_{B}=\varphi$.
\end{definition}

If $f\colon X \to Y$ is a continuous map, then a Sullivan model for $A_{PL}(f)$ is called a Sullivan model for $f$.

In the case of the morphism $\QQ\to (A, d)$, $1\mapsto 1$, this definition reduces to the definition  of a Sullivan model of $(A, d)$.

Suppose
$$\varphi_{0}, \varphi_{1}\colon (B\otimes \Lambda V,  d)\to (A, d)$$
are two morphisms of commutative cochain algebras, in which $ (B\otimes \Lambda V,  d)$ is a relative Sullivan algebra and $\varphi_{0}$ and $\varphi_{1}$ restrict to the same morphism $\alpha\colon (B, d)\to (A, d)$.

\begin{definition}\cite[Secton~14, page 185]{RHT}
$\varphi_{0}$ and $\varphi_{1}$ are homotopic rel $B$ ($\varphi_{0}\sim\varphi_{1}$ rel $B$) if there is a morphism
$$\Phi\colon (B\otimes \Lambda V,  d)\to (A, d)\otimes (\Lambda (t\oplus dt), d)$$
such that $(id.\varepsilon_{0})\Phi=\varphi_{0}$, $(id.\varepsilon_{1})\Phi=\varphi_{1}$ and $\Phi(b)=\alpha(b)\otimes1$, $b\in B$, where the augmentations  $\varepsilon_{0}, \varepsilon\colon \Lambda(t, dt)\to \QQ$ are defined by  $\varepsilon_{0}(t)=0$, $\varepsilon_{1}(t)=1$, respectively.

If $B=\QQ$, then we merely say that $\varphi_{0}$ and $\varphi_{1}$ are homotopic.
\end{definition}

Suppose
\[
\alpha\colon (A, d)\to (A' , d)
\]
is an arbitrary morphism of commutative cochain algebras which satisfy $H^{0}(A) =\QQ=H^{0}(A')$. Let $m\colon (\Lambda V,d)\to (A,d)$ and $m' \colon (\Lambda V',d)\to (A',d)$ be Sullivan models. There is a unique homotopy class of morphisms
$$\varphi\colon (\Lambda V, d)\to (\Lambda V', d)$$
such that $m'\varphi\sim \alpha m$.

\begin{definition}\cite[Secton~14, page 154]{RHT}
A morphism $\varphi\colon (\Lambda V, d)\to (\Lambda V', d)$ such that $m'\varphi\sim \alpha m$ is called a \emph{Sullivan representative} for $\alpha$.  If  $f\colon X\to Y$ is a continuous map, then a Sullivan representative of $A_{PL}(f)$ is called a Sullivan representative of $f$.
\end{definition}

Let
$$\varphi \colon (B,d)\to (C,d)$$
be a morphism of commutative cochain algebras such that $H^{0}(B) =\QQ=H^{0}(C)$ and $H^{1}(\varphi)$ is injective.

\begin{thm}\cite[Theorem~14.12]{RHT}
The morphism $\varphi$  has a minimal Sullivan model
$$m \colon (B\otimes \Lambda V,  d) \xrightarrow{\simeq} (C, d).$$
If $m' \colon (B\otimes \Lambda V',  d) \xrightarrow{\simeq} (C, d)$ is a second minimal Sullivan model for $\varphi$, then
there is an isomorphism
$$\alpha \colon (B\otimes \Lambda V,  d)\xrightarrow{\cong} (B\otimes \Lambda V',  d)$$
 restricting to $ id\mid_{B}$, and such that $m'\alpha\sim m \,\,\text{rel} \,\,B.$
\end{thm}

\begin{definition}\cite[Section~14, Page 182]{RHT}
A relative Sullivan algebra $(B\otimes \Lambda V,  d) $ is \emph{minimal} if
$$\im(d)\subseteq B^{+}\otimes \Lambda V+B\otimes \Lambda ^{\geq 2}V.$$
A minimal Sullivan model for $ \varphi \colon (B,d) \to (C,d) $ is a Sullivan model $ (B\otimes \Lambda V,  d) \xrightarrow{\simeq}(C, d) $ such that $(B\otimes \Lambda V,  d) $ is minimal.
\end{definition}

 \begin{definition}
A Sullivan algebra of the form $( \Lambda ( U \otimes dU ), d )$,  where $d\colon U\rightarrow dU$ is an isomorphism, is called contractible.
 \end{definition}

 \begin{thm}\cite[Theorem~14.9]{RHT}\label{T:Surjectivization}
 Let $(B\otimes \Lambda V,  d)$ be a relative Sullivan algebra. Then the identity of $B$ extends to an isomorphism of cochain algebras,
$$( B\otimes \Lambda W) \otimes ( \Lambda ( U \oplus dU ), d ) \xrightarrow{\cong} ( B\otimes \Lambda V),$$
in which $( B\otimes \Lambda W)$ is a minimal relative Sullivan algebra and $( \Lambda (U \oplus dU ), d )$ is contractible.
 \end{thm}

\begin{rem}
For a contractible algebra $( \Lambda (U \oplus dU ), d )$ we have that $H^{0}( \Lambda (U \oplus dU ), d )=\QQ$ and $H^{i}( \Lambda (U \oplus dU ), d )=0$ for $i\geq 1$.
\end{rem}

\begin{rem}\label{R:Surjective_Trick}
Recall that up to homotopy any morphism of cochain algebras can be made surjective via ``the surjective trick'' (see \cite[Page~148]{RHT}), i.e.~in the standard model category morphisms are homotopic to fibrations.
\end{rem}

After a quick review of the basics of rational homotopy theory, we recall the definitions of \emph{Sullivan spaces} and
\emph{rationally nilpotent spaces}.   We refer the reader to \cite[Chapters~7 and 8]{RHT_II} for more details.


Let $(X, \ast)$ be a path-connected space. Let  $(\Lambda V, d)$ be its minimal Sullivan model  with a quasi-isomorphism $m_X\colon (\Lambda V, d)\to A_{PL}(X)$.  Then 
as in \cite[Section~1.8]{RHT_II} one can construct  linear maps
$$\pi_k({m_X})\colon \pi_k(X)\otimes \QQ\to \pi_k(\Lambda V, d),\quad k\geq 2,$$
where the groups $ \pi_k(\Lambda V, d)$ are the homotopy groups of $(\Lambda V, d)$ (see \cite[Page~35]{RHT_II} for the definition). If $X$ is simply-connected and if $H^*(X, \QQ)$ is a graded space of finite type, i.e.   $\dim~H^i(X, \QQ)<\infty$, for all $i\geq 0$, then the linear maps $\pi_k({m_X})$, $k\geq 2$, are all isomorphisms.
This may no longer be true if $X$ is not simply-connected. When it is true, then the Sullivan model of $X$  reflects the homotopy invariants of $X$. Thus, we have the following definition.

\begin{definition} \cite[Page~195]{RHT_II}\label{D:Sullivan_Space}
A Sullivan space is a path--connected based space $(X,\ast)$ with universal covering space $(\tilde{X},\ast)$ such that
\begin{itemize}
\item[(i)]
 $\dim H^1(X, \QQ) <\infty$ and $\dim H^k(\tilde{X}, \QQ) <\infty$ for  $k \geq 2$;
 \item[(ii)]
A minimal model $m_X\colon (\Lambda V, d)\to A_{PL}(X)$ satisfies
 $$\pi_k({m_X})\colon \pi_k(X)\otimes \QQ\xrightarrow[]{\cong} \pi_k(\Lambda V, d)\quad k\geq 2.$$
  \end{itemize}
\end{definition}

Recall that if $G$ is a group, then the \emph{lower central series} of $G$ is defined by $G^1=G$ and $G^{n+1}=[G, G^n]$.

\begin{definition}\cite[Page~227]{RHT_II}   A group $G$ is called \emph{rationally nilpotent} if  \linebreak[4]$H_*(\B G^{n+2}, \mathbb{Q})=\mathbb{Q}$ for some $n\geq 0$.
\end{definition}

\begin{example}
\begin{itemize}
\item[i)]Every nilpotent group is rationally nilpotent.

\item[ii)] Let $G$ be a finite group. Then by \cite[Corollary~6.5.9]{Weibel}, group homology satisfies $H_{p}(G, \mathbb{Q})=0$, for all $p\neq 0$. Therefore, $G$ is rationally nilpotent, for  $H_{*}(\B G, \QQ)=H_{*}(G, \QQ)$.
\end{itemize}
\end{example}

\begin{definition} \cite[Page~227]{RHT_II} 
A connected CW complex $Y$ is called \emph{rationally nilpotent} if $\pi_{1}(Y, \ast) $ is a rationally nilpotent group which acts nilpotently on each $\pi_{n}(\tilde{Y})\otimes\QQ$, where $\tilde{Y}$ is the universal covering space of $Y$.
\end{definition}

The following Proposition shows that rationally nilpotent spaces are indeed Sullivan spaces.

\begin{prop}\label{P:Rationally_Nilpotent_is_Sullivan}\cite[Corollary~8.2]{RHT_II}
Let $(Y,\ast)$ be a connected CW complex  pointed by a $0$-cell. If $(Y,\ast)$ is a rationally nilpotent CW complex with rational cohomology of finite type then $Y$ is a Sullivan space.
\end{prop}

\begin{definition}\cite[Page~109]{RHT_II}
Let $(C, d)$ be a commutative cochain algebra for which $H^{0}(C)=\QQ$. A minimal Sullivan \emph{$k$-model} for $(C,d)$ is a morphism $\varphi\colon (\Lambda V, d) \to (C, d)$ from a minimal Sullivan algebra in which $ V = V^{\geq k }$ and $H^{i}(\varphi)$ is an isomorphism for $ i \leq k$ and injective for $i = k+1$. If $X$ is a path--connected topological space then a $k$-minimal Sullivan model for $A_{PL}(X)$ is called a minimal Sullivan $k$-model for $X$.
\end{definition}

Now we show that a Sullivan space rationally is the same as its universal cover, that is,  they have isomorphic minimal Sullivan models.
\begin{prop}\label{L_Sullivan_Model_of_Non_Simply_Connected}
Let $(X, \ast)$ be a connected Sullivan CW complex with finite fundamental group $\Gamma$, and let $\tilde{X}$ be its universal cover.   Then
\begin{itemize}
\item [(i)]
the minimal  Sullivan model of $X$ is isomorphic to the minimal Sullivan model of $\tilde{X}$.
\item[(ii)] for the universal covering map $p\colon \tilde{X}\to X$, the morphism $A_{PL}(p)$ induced by $p$ is a quasi-isomorphism.
\end{itemize}
\end{prop}

\begin{proof}
\begin{itemize}
\item[(i)]
 First note that since $\Gamma$ is a finite group, $H^1(\B\Gamma, \mathbb{Q})=0$.
 By \cite[Corollary~2.1 and Proposition~7.6]{RHT_II} we conclude that the minimal Sullivan  $1$-model of $\B\Gamma$ is isomorphic to $(\mathbb{Q}, 0)$. Therefore, by \cite[Theorem 7.2]{RHT_II}, the minimal Sullivan model of $X$ is isomorphic to the minimal Sullivan model of its universal covering space.
 \item[(ii)]  Suppose
$$F\xrightarrow[]{j}X\times_{\B\Gamma}M\B\Gamma\xrightarrow[]{q}\B\Gamma$$
is a classifying space fibration for a connected CW complex $(X, \ast) $ (see \cite[Page~93]{RHT_II} for the definition).  Then by \cite[Lemma~4.3]{RHT_II}, there is a commutative diagram

\begin{figure}[H]
\centering

\centering
\begin{tikzpicture}[]
\draw [->] (-1.6, 0) -- (-.6,0) ;
\node at (-1.7,0) [left ]{$\ \ \ \ (F, \ast)$};
\node at (.6,0)  {$\ \ \ \ (X\times_{\B\Gamma}M\B\Gamma, \ast)$};
\draw [->] (-1.6,-2) -- (-0.6,-2) ;
\node at (-1.7,-2) [left ]{$\ \ \ \ (\tilde{X},\ast)$};
\node at (0,-2)  {$\ \ \ \ (X, \ast),$};
\draw [->] (-2.3, -.3) -- (-2.3,-1.7) ;
\draw [->] (0.5, -.3) -- (0.5,-1.7) ;
\node at (-1.4, 0) [above ]{$\ \ \ \ j$};
\node at (-1.4,-2) [below ]{$\ \ \ \ p$};
\node at (0.1, -1) [right ]{$\ \ \ \ \pi$};
\node at (-2.3,-1) [left]{$\ \ \ \ \sigma$};
\end{tikzpicture}
\end{figure}
\noindent where  $\pi$ and $\sigma$ are weak homotopy equivalences. By Part (i) and \cite[Diagram~(7.4), Page~197]{RHT_II} we have that $A_{PL}(j)$ is a quasi-isomrphism, and therefore, $A_{PL}(p)$ is a quasi-isomorphism.
\end{itemize}

\end{proof}

\begin{cor}\label{C_Sullivan_Model_of_BH}
Let $H$ be a compact Lie group whose classifying space $\B H$ is a Sullivan space. Then the minimal Sullivan model of $\B H$ is isomorphic to the minimal Sullivan model of $\B H_0$, and the map $A_{PL}(B\iota)$ is a quasi-isomorphism,  where $H_0$ is the identity component of $H$, and $\iota \colon H_{0}\to H$ is the inclusion.
\end{cor}

\begin{proof}
Note that  $B\iota\colon \B H_{0}\to \B H$ is the universal covering map of $\B H$. The result now follows from Proposition~\ref{L_Sullivan_Model_of_Non_Simply_Connected}.
\end{proof}

\begin{prop}\label{P:Covering_of_a_Sullivan_space}
Let $(X, \ast)$ be a connected Sullivan CW complex with finite fundamental group. Then any covering of $(X, \ast)$ is a Sullivan space.
\end{prop}

\begin{proof}
Let $(\bar{X}, {\ast})$ be a covering of $(X, \ast)$ and $(\tilde{X}, \ast)$ be their universal cover.  Since $\pi_1 (X, \ast)$ is finite, so is $G=\pi_1(\bar{X}, {\ast})$. Hence  a Sullivan $1$-model for the classifying space of the finite group $G$ is  $(\QQ, 0)$. Moreover, since $H^k(\B G, \QQ)=0$, for all $k\geq 1$,  a Sullivan $1$-model for $\B G$ is indeed its minimal Sullivan model. Furthermore, the action by covering transformations of  $G$ in $H^{*}(\tilde{X}, \QQ)$ is just the subaction of $\pi_1 (X, \ast)$ via covering transformations  in $H^{*}(\tilde{X}, \QQ)$, which is locally nilpotent since $(X, \ast)$ itself is a Sullivan space. Now the result follows from \cite[Theorem~7.1]{RHT_II}.
\end{proof}

Now we give a Sullivan model for a homogeneous space $G/H$ with $G$ connected and with the inclusion $\iota\colon H\to G$. When $H$ is connected as well, this model is known (see for example \cite[Proposition~15.16]{RHT}). The proof  is essentially the same as in the connected case.

\begin{prop}\label{P:Sullivan_model_of_a_homogeneous_space}
Let $G$ be a compact connected Lie group and $H$ be a closed subgroup of $G$ whose classifying space $\B H$ is a Sullivan space. Then a Sullivan model for $G/H$ is given by
\[
(\Lambda(V_{\B H}\oplus V_{G}), d)
\]
where  $(\Lambda V_{\B H}, 0)$ is the minimal Sullivan model for $\B H$ and $(\Lambda V_{G}, 0)$ with $V_G=\langle v_1,\ldots, v_k\rangle$ is the minimal Sullivan model for $G$. The differential $d$---then extending it as a derivation---is  defined by  $d|_{\Lambda V_{\B H}}=0$ and
\[
d(v_i)=H^*(\B\iota)(x_i)
\]
where $\Lambda V_{\B G}=\Lambda \langle x_1,\ldots, x_k\rangle$ and  where we identify $v_i = sx_i$  with $x_i$ via a $(-1)$-degree shift.
\end{prop}

\begin{proof}
Let
\[
G\to \E G\to \B G
\]
be the universal principal $G$-bundle, and $\B\iota\colon \B H\to \B G$ be the map  induced by the inclusion  $\iota\colon H\to G$. Then the pullback of this  principal $G$-bundle along $\B\iota$ is isomorphic to

\begin{equation}\label{EQ:Principal_bundle}
G\to \frac{\E H\times G}{H}\to \B H.
\end{equation}
Further, the projection map

\[
q\colon \frac{\E H\times G}{H}\to G/H
\]
is a weak homotopy equivalence \cite[Proposition~2.9]{RHT}. Therefore, every Sullivan model for $\frac{\E H\times G}{H}$ is also a Sullivan model for $G/H$.

By \cite[Example~2.72]{FeOpTa}, a model for the principal bundle \eqref{EQ:Principal_bundle} is given by
\begin{equation}\label{Eq:Model_Principal_Bundle}
(\Lambda V_{\B H}, d)\to (\Lambda  V_{\B H}\oplus \Lambda V_{G}, d)\to (\Lambda V_{G}, 0).
\end{equation}
Now we need to describe the differential $d$ for $(\Lambda  V_{\B G}\oplus \Lambda V_{G}, d)$. First note that by Corollary~\ref{C_Sullivan_Model_of_BH}, we have  $(\Lambda V_{\B H}, d)\cong (H^*(\B H, \mathbb{Q}), 0)$. Thus $d|_{\Lambda V_{\B H}}=0$. Let $V_G=\langle v_1,\ldots, v_k\rangle$, and $\Lambda V_{\B G}=\langle x_1,\ldots, x_k\rangle$ where we identify $v_i = sx_i$  with $x_i$ via a $(-1)$-degree shift. Then by \cite[Example~2.72]{FeOpTa}, $d(v_i)=H^*(\B\iota)(x_i)$ as desired.
\end{proof}

Recall that the cohomology of  homogeneous spaces $G/H$ with positive Euler characteristic and with both $G$ and $H$ being connected is concentrated in even degrees (see for example \cite[Section~13, Theorem~2]{On}).  By  Proposition~\ref{P:Sullivan_model_of_a_homogeneous_space} we can easily generalize this fact as follows:

\begin{prop}\label{P:Cohomology_Homogeneous_Sp}
Let $G$ be a compact connected Lie group and $H$ be a closed subgroup of $G$ whose classifying space $\B H$ is a Sullivan space. Assume additionally that $\rank~G=\rank~H$. Then the cohomology of the homogeneous space $G/H$ is concentrated in even degrees.
\end{prop}

\begin{proof}
By Proposition~\ref{P:Sullivan_model_of_a_homogeneous_space}, the minimal Sullivan model for $G/H$ is pure (see \cite[p.~435]{RHT}) with as many generators in odd degrees as in even degrees, since $\rank~G=\rank~H$. The result follows by \cite[Proposition 32.2]{RHT}. 
\end{proof}

\begin{prop}\cite[Proposition~8.1]{RHT_II}\label{P:Rationaly_Nilpotent_Equivalent_Def}
A connected CW complex $(Y,\ast)$ pointed by a $0$-cell is rationally nilpotent if and only if $\pi_1(Y,\ast)$ is rationally nilpotent and the representation of $\pi_1(Y,\ast)$ induced by covering transformations in each $H_k(\tilde{Y}, \ast)$ is nilpotent.
\end{prop}

Let us use the shorthand notation PCHS for a ``positively curved homogeneous space" in the following.
\begin{prop}\label{P:non-simply-connected_homogeneous_space_Rationally_Nilpotent}
Let $K/H$ be an odd dimensional PCHS.
Then $K/H$ is rationally nilpotent.
\end{prop}

\begin{proof}
If $K/H$ is simply-connected, then it is nilpotent and, in particular, rationally nilpotent. Hence, we assume that $K/H$ is not simply-connected. By the classification of PCHSs, an odd dimensional non-simply-connected PCHS is covered by $S^{2n+1}$ or by an Aloff-Wallach space $W^{7}_{p,q}$.  First note that by the Lefschetz Fixed Point Theorem, every homeomorphism $f\colon S^{2n+1}\to S^{2n+1}$ which does  not have a fixed point is orientation-preserving. Therefore, every  covering transformation of the covering $S^{2n+1}\to S^{2n+1}/\Gamma$ is homotopy equivalent to the identity, and hence induces the identity on homology groups. By Proposition~\ref{P:Rationaly_Nilpotent_Equivalent_Def} we conclude that $S^{2n+1}/\Gamma $ is rationally nilpotent since $\Gamma$ is finite.  In  the case that the universal covering space  is an Aloff--Wallach space, the classification of the isometry groups of simply-connected positively curved homogeneous spaces  \cite[Subsections~4.8 and 4.9]{Shankar} shows that the deck transformation groups of the coverings $W^{7}_{p,q}\to W^{7}_{p,q}/\Gamma$ are contained  in the identity component of the isometry groups. Therefore, each covering transformation is homotopic to the identity and hence induces the  identity on homology groups. Similarly, since the fundamental group is finite,  Proposition~\ref{P:Rationaly_Nilpotent_Equivalent_Def} gives that $W^{7}_{p,q}/\Gamma$ is rationally nilpotent.
\end{proof}

\begin{rem}
One can prove Proposition~\ref{P:non-simply-connected_homogeneous_space_Rationally_Nilpotent} without using Proposition~\ref{P:Rationaly_Nilpotent_Equivalent_Def} as follows. First note that since the fundamental group of a PCHS is finite, it is rationally nilpotent. Therefore, we only need to prove that the action of the fundamental group of a PCHS on its universal covering (via covering transformations) is nilpotent. As pointed out in the proof of  Proposition~\ref{P:non-simply-connected_homogeneous_space_Rationally_Nilpotent}, each covering transformation is freely homotopic to the identity (in fact $N_{G}(H)/H$ is contained in the identity component of the isometry group of $G/H$---see  \cite{Shankar}). Let $f\colon (\tilde{X}, \tilde{x}_0)\to (\tilde{X}, \tilde{x}_1)$ be a covering transformation corresponding to $[\gamma]\in \pi_1(X, x_0)$. Let $\tilde{\gamma}$ be the lift of $\gamma$  under the universal covering map emanating from $\tilde{x}_0$.  Thus $\tilde{\gamma}$ connects $\tilde{x}_0$ to $\tilde{x}_1$. Now using $\tilde{\gamma}$, we can define an isomorphism
\begin{align*}
\tilde{\gamma}_*\colon \pi_n(\tilde{X}, \tilde{x}_1)&\to \pi_n(\tilde{X}, \tilde{x}_0)\\
[\beta] &\mapsto [\tilde{\gamma}.\beta],
\end{align*}
where $\tilde{\gamma}.\beta\colon (S^n, s_0)\to (\tilde{X}, \tilde{x}_0)$ is the map defined in \cite[Page~341]{Hatcher}. Notice that $\tilde{\gamma}.\beta$ and $\beta$ are freely homotopic. Now define the action of the group of covering transformations on $\pi_n(\tilde{X}, \tilde{x}_0)$ via the map $\tilde{\gamma}_*\circ f_{\sharp}$.
We want to show that $\tilde{\gamma}_*\circ f_{\sharp}$ is the identity map of $\pi_n(\tilde{X}, \tilde{x}_0)$. By \cite[Proposition~4A.2]{Hatcher}, there is a bijection
\[
\Phi\colon \pi_n(\tilde{X}, \tilde{x}_0)\to [S^n, \tilde{X}],
\]
where $[S^n, \tilde{X}]$ is the space of free homotopy classes of maps from $S^n$ to $\tilde{X}$. Since $f$ is freely homotopic to the identity, we have that $f\circ \alpha$ is freely homotopic to $\alpha$ for $\alpha\in \pi_n(\tilde{X}, \tilde{x}_0)$. Thus $\Phi([\alpha])=[f\circ \alpha]$. Further, $\tilde{\gamma}.f\circ \alpha$ is freely homotopic to $f\circ \alpha$, and $\Phi([\tilde{\gamma}.f\circ \alpha])=[f\circ \alpha]$. Since $\Phi$ is a bijection, we have that $[\tilde{\gamma}.f\circ \alpha]=[\alpha]$ in $\pi_n(\tilde{X}, \tilde{x}_0)$. That is $\tilde{\gamma}_*\circ f_{\sharp}[\alpha]=[\tilde{\gamma}.f\circ \alpha]=[\alpha]$.
\end{rem}

\begin{prop}\label{L:Even_Dim_Nilpotency}
Non-simply-connected even-dimensional positively curved Riemannian manifolds are not nilpotent.
\end{prop}

\begin{proof}
Let $M^{2n}$ be a non-simply--connected positively curved Riemannian manifold. By Synge's Theorem $M^{2n}$ is not orientable and $\pi_{1}(M)=\mathbb{Z}_{2}$. Let $g\in \pi_{1}(M)$  be the non-trivial  element.  Then  by the Weinstein Theorem the covering transformation $\mu_{g}\colon \tilde{M}\to \tilde{M}$ is an orientation--reversing isometry, where $\tilde{M}$ is the universal covering of $M$.  Consider the fibration
$$(F, \ast)\to (M\times_{B}MB, \ast)\to  (B, \ast)$$
as in \cite[Section~3.2]{RHT_II}, where $B=\B  \pi_{1}(M, \ast)$ and $F$ is weakly equivalent to $\tilde{M}$. By \cite[Lemma~4.3]{RHT_II} and the fact that $\mu_{g}$ is an orientation-reversing map we conclude that the action of $\pi_{1}(M, \ast)$ on $H_{*}(\tilde{M})$ is not nilpotent. Therefore, by \cite[Lemma~2.18]{HiMiRo}, $M$ is not nilpotent (cf.  \cite[Page~80]{FeOpTa}).
\end{proof}

\begin{prop}\label{L:Even_Dim_Sullivan}
Let $M=K/H$ be a  non-simply-connected even-dimensional positively curved homogeneous space. Then at least one of the spaces  $\B K$  and  $\B H$ is not a Sullivan space.
\end{prop}

\begin{proof}
By contradiction assume that both $\B K$  and  $\B H$ are Sullivan spaces. Then by Proposition~\ref{P:Covering_of_a_Sullivan_space}, $\B (K_0\cap H)$, being a covering space of $\B H$, is a Sullivan space.  Since  $M=K_0/K_0\cap H$, we can apply   Proposition~\ref{P:Sullivan_model_of_a_homogeneous_space} and Corollary~\ref{C_Sullivan_Model_of_BH} to obtain a Sullivan model for $M$ given by
$(\Lambda V_{\B H_0}\otimes \Lambda V_{K_0}, d)$.
Hence $H^{*}(K_{0}/H_{0}, \QQ)=H^{*}(M, \QQ)$.
Further, since $K_0/H_0$ is a covering space of $K_0/K_0\cap H=M$, it is either homeomorphic to $M$ itself or to its universal covering space $\tilde{M}$, as $\pi_1(M)=\ZZ_2$.  

Either case, however,  yields a contradiction: If $K_0/H_0=M$, then $M$ is a simple space \cite[Proposition~1.62 and Example~1.63]{FeOpTa} and therefore, nilpotent, contrary to Proposition~\ref{L:Even_Dim_Nilpotency}. If $K_0/H_0=\tilde M$, we obtain  that $H^{*}(\tilde{M}, \QQ)=H^{*}(M, \QQ)$.
\end{proof}


\section{ A model for a double mapping cylinder}\label{S_ Sullivan_model_for_double_mapping_cylinder}

Let $X$ be  
a generalized cohomogeneity one space with  group diagram $(G, H, K^{-}, K^{+})$ as in Decomposition \eqref{Eq:decomposition_2}. In this section based on \cite{GroveHalperin} and \cite[Chapter~13]{RHT},  we  present  a commutative differential graded algebra model for  the Borel construction $X_{G}=\frac{\E G\times X}{G}$ and then compute the equivariant cohomology of $X$, i.e.,  $H^{*}(X_{G}, \QQ)$, which is isomorphic  to the cohomology of this model. Indeed, we show that $X_{G}$ can be expressed as a \emph{double mapping cylinder}.  Let us construct a model for such a double mapping cylinder in the following.
We refer the reader to \cite{GroveHalperin} and \cite[Chapter~13]{RHT} for more details.

\vspace{3mm}

Note that the fibers $ K^{\pm}/H$ in the definition of a generalized cohomogeneity one space are not necessarily connected. However, from now on we assume that $K^{\pm}/H$ are connected, unless stated otherwise.

\vspace{3mm}

Let $(Z, Y)$ be a topological pair and $f\colon Y\to W$ be a continuous map. Let
$W\cup_{f}Z$ be the \emph{adjunction space|} by attaching $Z$ to $W$ along $f$. Now consider the following diagram.

\begin{equation}\label{Eq:fiber_Product}
A_{PL}(Z)\xrightarrow{A_{PL}(i)}A_{PL}(Y)\xleftarrow{A_{PL}(f)}A_{PL}(W).
\end{equation}

Let
$$A_{PL}(Z)\times_{A_{PL}(Y)}A_{PL}(W)=\{(\alpha, \beta)\in A_{PL}(Z)\times A_{PL}(W)\mid A_{PL}(i)(\alpha)= A_{PL}(f)(\beta)\}$$
be the \emph{fiber product} corresponding to Equation~\eqref{Eq:fiber_Product}. Now we determine the  relation between $A_{PL}(Z)\times_{A_{PL}(Y)}A_{PL}(W)$ and $A_{PL}(W\cup_{f}Z)$ as   in \cite{RHT}. First consider the following diagram

\begin{figure}[H]
\centering
\centering
\begin{tikzpicture}[]
\draw [->] (-1.6, 0) -- (-.6,0) ;
\node at (-1.7,0) [left ]{$\ \ \ \ Y$};
\node at (-.4,0)  {$\ \ \ \ Z$};
\draw [->] (-1.6,-2) -- (-0.6,-2) ;
\node at (-1.7,-2) [left ]{$\ \ \ \ W$};
\node at (0,-2)  {$\ \ \ \ W\cup_{f}Z,$};
\draw [->] (-2, -.3) -- (-2,-1.7) ;
\draw [->] (-.2, -.3) -- (-0.2,-1.7) ;
\node at (-1.4, 0) [above ]{$\ \ \ \ i$};
\node at (-1.4,-2) [below ]{$\ \ \ \ i_{W}$};
\node at (-0.6, -1) [right ]{$\ \ \ \ f_{Z}$};
\node at (-2.1,-1) [left]{$\ \ \ \ f$};
\end{tikzpicture}
\end{figure}
which induces the diagram
\begin{figure}[H]
\centering
\centering
\begin{tikzpicture}[]
\draw  [->] (-.6,0) --  (-1.6, 0) ;
\node at (-1.7,0) [left ]{$\ \ \ \ A_{PL}(Y)$};
\node at (0,0)  {$\ \ \ \ A_{PL}(Z)$};
\draw [->] (-0.6,-2) -- (-1.6,-2) ;
\node at (-1.7,-2) [left ]{$\ \ \ \ A_{PL}(W)$};
\node at (.5,-2)  {$\ \ \ \ A_{PL}(W\cup_{f}Z).$};
\draw [->]  (-2,-1.7) -- (-2, -.3) ;
\draw [->]  (-0.2,-1.7) -- (-.2, -.3) ;
\end{tikzpicture}
\label{F:APL_of_Adjunction_Diagram}
\end{figure}
Then we have the following morphism
$$(A_{PL}(i_{W}), A_{PL}(f_{Z}))\colon A_{PL}(W\cup_{f}Z)\to A_{PL}(Z)\times_{A_{PL}(Y)}A_{PL}(W).$$

\begin{prop}\cite[Proposition~13.5]{RHT}\label{P:fiber_product_adjunction_space}
If $H_{*}(Z, Y; \QQ)\cong H_{*}(W\cup_{f}Z, W; \QQ)$, then the morphism
$(A_{PL}(i_{W}), A_{PL}(f_{Z}))$ is a quasi-isomorphism. Thus the fiber product is a commutative model for the adjunction space.
\end{prop}

Now we pass to Sullivan models. Suppose that the spaces $Z, Y, W$ are path--connected and

\begin{align*}
m_{W}\colon \Lambda V_{W}\to A_{PL}(W),\quad
m_{Z}\colon \Lambda V_{Z}\to A_{PL}(Z),\quad
m_{Y}\colon \Lambda V_{Y}\to A_{PL}(Y),
\end{align*}
are Sullivan models and
$$
\Lambda V_{W}\xrightarrow{\phi}\Lambda V_{Y}\xleftarrow{\psi}\Lambda V_{Z},
$$
are Sullivan representatives for $f$ and $i$.

\begin{prop}\cite[Proposition~13.6]{RHT}\label{P:Fiber_Products_Sullivan_Spaces}
If $H_{*}(Z, Y; \QQ)\cong H_{*}(W\cup_{f}Z, W; \QQ)$,  and one of the morphisms $\phi$ or $\psi$ is surjective, then $\Lambda V_{W}\times_{\Lambda V_{Y}}\Lambda V_{Z}$ is a commutative model for $W\cup_{f}Z$.
\end{prop}

Consider the following commutative diagram  of commutative cochain algebras

\begin{figure}[H]
\centering

\centering
\begin{tikzpicture}[]
\draw [->] (-1.6, 0) -- (-.5,0) ;
\draw [->] (2.5, 0) -- (1.4,0) ;
\node at (-1.7,0) [left ]{$\ \ \ \ (C, d)$};
\node at (.2,0)  {$\ \ \ \ (B, d)$};
\node at (2,0) [right] {$\ \ \ \ (A, d)$};
\draw [->] (-1.6,-2) -- (-0.5,-2) ;
\draw [->] (2.5, -2) -- (1.4,-2) ;
\node at (-1.7,-2) [left ]{$\ \ \ \ (C', d)$};
\node at (.2,-2)  {$\ \ \ \ (B', d)$};
\node at (2,-2) [right] {$\ \ \ \ (A', d).$};
\draw [->] (-2.3, -.3) -- (-2.3,-1.7) ;
\draw [->] (0.3, -.3) -- (0.3,-1.7) ;
\draw [->] (3, -.3) -- (3,-1.7) ;
\node at (-1.4, 0) [above ]{$\ \ \ \ \phi$};
\node at (1.6, 0) [above ]{$\ \ \ \ \psi$};
\node at (-1.4,-2) [below ]{$\ \ \ \ \phi'$};
\node at (1.6,-2) [below ]{$\ \ \ \ \psi'$};
\node at (-0.7, -1) [right ]{$\ \ \ \ \beta$};
\node at (2.5, -1) [right ]{$\ \ \ \ \alpha$};
\node at (-2.3,-1) [left]{$\ \ \ \ \gamma$};

\end{tikzpicture}
\end{figure}

The following Lemma shows the relation between the fiber products  $C\times_{B}A$ and
$C'\times_{B'}A'$.

\begin{lem}\cite[Lemma~13.3]{RHT}\label{L:Equivalence_of_two_fiber_product}
If $\beta, \gamma, \alpha$ are quasi-isomorphisms and if one of $\phi, \psi$ and one of $\phi', \psi'$ are surjective, then $(\gamma, \alpha)\colon C\times_{B}A\to
C'\times_{B'}A'$ is a quasi-isomorphism.

\end{lem}

Let $X$ be a generalized cohomogeneity one space given by the group diagram  \linebreak[4]$(G, H, K^-, K^+)$. We now describe a commutative model for $X_{G}$.  First we recall the  construction of a model for
$$A_{1}\xrightarrow{\phi}A\xleftarrow{\psi}A_{2}$$
according to \cite{GroveHalperin}.  We define
\begin{equation}\label{Eq:Grove_Halperin_Model}
D\big((A_{1}, d)\xrightarrow{\phi}(A, d)\xleftarrow{\psi}(A_{2}, d)\big)\colon =(\{(a_{1}, a_{2})\in A_{1}\otimes C\oplus A_{2}\mid \Phi(a_{1})=\psi(a_{2})\}, d),
\end{equation}
where $C$ is a contractible algebra and $\Phi$ is a surjective map induced by $\phi$ (see Remark \ref{R:Surjective_Trick}), and $d$ is the obvious differential.

\begin{prop}\label{P:X_G as a double mapping cylinder}
Let $X$ be a generalized cohomogeneity one space with  group diagram $(G, H, K^{-}, K^{+})$. Then $X_G$ is homeomorphic to the double mapping cylinder of the following fibrations:
\[
K^{\pm}/H\to \B H\to \B K^{\pm}.
\]
\end{prop}

\begin{proof}
By definition,  $X$ is the union of two cone bundles over the two singular orbits, that is,
\begin{align*}
X=G\times_{K^-} C(\frac{K^-}{H})\bigcup_{G/H}G\times_{K^+} C(\frac{K^+}{H}).
\end{align*}
Hence, $X_G$   is the union of two fiber bundles over  $\B K^{\pm}$ whose fibers are cones over $K^{\pm}/H$. That is,
\begin{align*}
X_G\cong \E G\times_{K^-} C(\frac{K^-}{H})\bigcup_{\B H} \E G\times_{K^+} C(\frac{K^+}{H}).
\end{align*}
Therefore, $X_G$ is the double mapping cylinder of the fibrations
\[
K^{\pm}/H\to \B H\to \B K^{\pm}.
\]
\end{proof}

\begin{prop}\label{P:Equivariant_Cohomology}
Let $X$ be a generalized cohomogeneity one space with  a group diagram $(G, H, K^{-}, K^{+})$,
 where the classifying spaces $\B H$, $\B K^-$ and $\B K^+$ are Sullivan spaces.
Then the equivariant cohomology of $X$ is isomorphic to the cohomology  of the following model:
\begin{equation} \label{Eq:Double_mapping_Algebra}
\tilde{D}=D\big((H^{*}(\B K^-_{0}, \QQ), 0)\xrightarrow[]{H^*(\B\iota^-)} (H^{*}(\B H_{0}, \QQ), 0)\xleftarrow[]{H^*(\B\iota^+)}(H^{*}(\B K^+_{0}, \QQ), 0)\big),
\end{equation}
where $\iota^{\pm}$ are the inclusions
$\iota ^{\pm}\colon H_{0}\to K^{\pm}_{0}$.

\end{prop}

\begin{proof}
By Proposition~\ref{P:X_G as a double mapping cylinder},

\begin{align*}
X_G\cong \E G\times_{K^-} C(K^-/H)\bigcup_{\B H} \E G\times_{K^+} C(K^+/H).
\end{align*}

By excision we have
$$H_{*}\big(\E G\times_{K^-} C(K^-/H), \B H;\QQ\big)\cong H_{*} \big(X_{G}, \E G\times_{K^+} C(K^+/H); \QQ\big).$$
Therefore, by Proposition~\ref{P:fiber_product_adjunction_space}, the fiber product \[A_{PL}\big(\E G\times_{K^-} C(K^-/H)\big)\times_{A_{PL}(\B H)}A_{PL}\big(\E G\times_{K^+} C(K^+/H)\big)\]
is a commutative model for $A_{PL}(X_{G})$. Further, $G/K^{\pm}$
is a deformation retract of \linebreak $\E G\times_{K^{\pm}} C(K^{\pm}/H)$. If we make the morphism $A_{PL}(\B\iota^{-} )\colon A_{PL}(\B K^{-})\to A_{PL}(\B H)$ surjective (see Remark \ref{R:Surjective_Trick}), then by Lemma~\ref{L:Equivalence_of_two_fiber_product}, the fiber product, $A_{PL}(\B K^{-})\times_{A_{PL}(\B H)}A_{PL}(\B K^{+})$ is quasi-isomorphic to $A_{PL}(X_{G})$. On the other hand, by Corollary~\ref{C_Sullivan_Model_of_BH},  the morphisms $A_{PL}(\B i^{\pm} )\colon A_{PL}(\B K^{\pm})\to A_{PL}(\B K^{\pm}_{0})$ and $A_{PL}(\B i) \colon A_{PL}(\B H)\to A_{PL}(\B H_{0})$ are quasi-isomorphisms. Again by Lemma~\ref{L:Equivalence_of_two_fiber_product}, and after making the morphism $A_{PL}(\B i^{-} )$ surjective, we conclude that  $A_{PL}(\B K^{-}_{0})\times_{A_{PL}(\B H_{0})}A_{PL}(\B K^{+}_{0})$ is quasi-isomorphic to $A_{PL}(X_{G})$.
Now consider the Sullivan models for $\B H_{0}$ and $\B K^{\pm}_{0}$,
which are $(H^{*}(\B H_{0},\QQ), 0)$,  $(H^{*}(\B K^{\pm}_{0},\QQ), 0)$, respectively. Then we use Proposition~\ref{P:Fiber_Products_Sullivan_Spaces} to conclude the result.
\end{proof}

\begin{rem}
From now on, we use $\tilde{D}$ to refer to the commutative cochain algebra in \eqref{Eq:Double_mapping_Algebra}, and we use $H(\tilde{D})$ to refer to the rational cohomology of $\tilde{D}$, i.e., the equivariant cohomology of the cohomogeneity one space.
\end{rem}

\section{Cohen--Macaulay modules }\label{S:CM}
In this section we collect some basic information about Cohen--Macaulay modules in general and Cohen--Macaulay actions in particular. We refer the reader to \cite{BrHe, GoTo, GoRo, GoMa14, GoMa18} for more detailed information.

Let  $A$ be a module over a  local Noetherian  ring $R$ with the maximal ideal $\mathfrak{m}$. Define
$$\Kdim A:= \Kdim \frac{R}{\Ann(A)},$$
and
$$\depth A:= \grad(\mathfrak{m}, A),$$
where $\Kdim$ stands for the \emph{Krull dimension}, $\Ann(A)=\{r\in R\mid ra=0\,\, \text{for all} \, \, a\in A\}$ and $\grad(\mathfrak{m}, A)$ is the length of a maximal \emph{$A$-regular sequence} in $\mathfrak{m}$ \cite[Chapter~1]{BrHe}.
Note that we always have $\depth A \leq\Kdim A$ \cite[Proposition~1.2.12]{BrHe}.

\begin{definition}\cite[Definition~2.1.1]{BrHe}
 Let $R$ be a Noetherian local ring.  A finite $R$-module $M\neq 0$
is a  \emph{Cohen--Macaulay} module if $\Kdim A=\depth A$. If $R$ itself is a Cohen--Macaulay module, then it is called a Cohen--Macaulay ring.
\end{definition}

In general, if $R$ is an arbitrary Noetherian ring, then $M$ is a Cohen--Macaulay module if the localization $M_{\mathfrak{m}}$ is a Cohen--Macaulay module over $R_{\mathfrak{m}}$ for all maximal ideals $\mathfrak{m}\in \Supp M$.

In the context of graded modules over graded rings, there is a graded analogue of local rings:

\begin{definition}\cite[Definition~1.5.13]{BrHe}
 Let $R$ be a graded ring. A graded ideal $\mathfrak{m}$ of $R$ is called $\prescript{*}{}{\text{maximal}}$  if every graded ideal that properly contains $\mathfrak{m}$ equals $R$. The ring $R$ is called $\prescript{*}{}{\text{local}}$ if it has a unique $\prescript{*}{}{\text{maximal}}$ ideal $\mathfrak{m}$.  A $\prescript{*}{}{\text{local}}$ ring with $\prescript{*}{}{\text{maximal}}$ ideal $\mathfrak{m}$ will be denoted by $(R, \mathfrak{m})$.
\end{definition}

\begin{example}\cite[Example~1.5.14, (b)]{BrHe}
 Let $R$ be a non--negatively  graded ring for which the subring of zero--degree elements  $R_0$ is a local ring with maximal ideal  $\mathfrak{m_0}$. Then $R$ is a $\prescript{*}{}{\text{local}}$ ring with $\prescript{*}{}{\text{local}}$ ideal $\mathfrak{m}=\mathfrak{m_0}\oplus(\oplus_{n>0}R_n)$.  In particular,  a non--negatively  graded algebra over  a field is $\prescript{*}{}{\text{local}}$.
\end{example}

Analogously, one can define the Krull dimension and the depth of a graded module $A$ over a $\prescript{*}{}{\text{local}}$ ring $(R, \mathfrak{m})$ as for a module over a local ring. That is, the Krull dimension of $A$ is the Krull dimension of the ring $\frac{R}{\Ann(A)}$, which is  the supremum of the \emph{heights} of prime ideals in $R$ containing $\Ann(A)$ and
$$\depth A:= \grad(\mathfrak{m}, A).$$
Let us point out that in the definition of Krull dimension the supremum is taken over all prime ideals containing $\Ann(A)$ not just the graded ones.

\begin{definition}\label{D:CM_Star_Local}
 A finitely generated graded module $A$ over a Noetherian graded $\prescript{*}{}{\text{local}}$ ring $R$ is Cohen--Macaulay if  $\Kdim A=\depth A$.
\end{definition}

Now a question that may arise is whether one can use either definition of Cohen--Macaulay module, i.e. one for a general Noetherian ring and one for the  $\prescript{*}{}{\text{local}}$ ring $(R, \mathfrak{m})$, and still get the same result. If $R/\mathfrak{m}$ is a field, it is indeed the case as \cite[Proposition~5.1]{GoTo} shows:

\begin{prop}\cite[Proposition~5.1]{GoTo}\label{No_matter_which_definition_of_CM_we_use}
Let $A$ be a finitely generated graded module over a Noetherian graded $\prescript{*}{}{\text{local}}$ ring $R$ with $\prescript{*}{}{\text{maximal}}$  ideal $\mathfrak{m}$ such that $R/\mathfrak{m}$ is a field. Then the following conditions are equivalent.
\begin{itemize}
\item[(i)] $A$ is Cohen--Macaulay over $R$ as in Definition~\ref{D:CM_Star_Local}.
\item[(ii)] $A_{\mathfrak{m}}$ is Cohen--Macaulay over the local ring $R_{\mathfrak{m}}$.
\item[(iii)]  $A_{\mathfrak{m}'}$ is Cohen--Macaulay over the local ring $R_{\mathfrak{m}'}$ for all (not necessarily graded) maximal ideals $\mathfrak{m}'\subset R$.
\end{itemize}
 If these conditions are satisfied, then the Krull dimensions of the $R$-module $A$ and the $R_{\mathfrak{m}}$-module $A_{\mathfrak{m}}$ coincide.
\end{prop}
Let $G$ be a compact connected Lie group which acts continuously on a topological space $X$ with the associated bundle  map $\pi \colon X_{G}\to \B G$. Then the homomorphism
$$
H(A_{PL}(\pi))\colon H^{*}(\B G,  \QQ)\to H^{*}(X_{G}, \QQ)
$$
induces an $H^{*}(\B G, \QQ)$-module structure on $H^{*}_{G}(X):=H^{*}(X_{G}, \QQ)$ (see for example \cite[Chapter~III]{Hsiang}).

\begin{definition}
We call the $G$-space $X$ \emph{Cohen--Macaulay} if its equivariant cohomology \linebreak[4]$H_G^*(X;\QQ)$ is an $H^{*}(\B G, \QQ)$-Cohen--Macaulay module.
\end{definition}

\noindent\textbf{Examples and Non-examples.}
\begin{itemize}
\item[(i)] Cohomogeneity one (smooth or topological) manifolds  are Cohen--Macaulay spaces \cite{GoMa14,GoMa18}.
\item[(ii)] \label{PageEx}Let $W^7_{1, 1}=\frac{\SU(3)}{S^1}$, where $S^1=\diag\{e^{i\theta}, e^{i\theta}, e^{-2i\theta}\}\subseteq \SU(3)$, be an Aloff--Wallach space. It is a positively curved homogeneous space. Hence, by Proposition~\ref{P:Suspension-action}, $\Susp(W^7_{1,1})$ is a cohomogeneity one Alexandrov space with  group diagram $(\SU(3), S^1, \SU(3), \SU(3))$. Note that $\Susp(W^7_{1,1})$ is not a cohomogeneity one (topological) manifold (see \cite{GGZ1} for more details). We claim that the equivariant cohomology of $\Susp(W^7_{1,1})$ as an $H^*(\B G)$-module is not Cohen--Macaulay. Indeed, by Lemma~\ref{L: Dim_of_Ring_Instead_of_Module}, we only need to show that the Krull dimension of the ring $H(\tilde{D})$ is not equal to its depth.  We have

\begin{align*}
\tilde{D}&=D(H^*(\B \SU(3), \QQ) \xrightarrow{\phi} H^* (\B S^1, \QQ) \xleftarrow{\psi} H^* (\B \SU(3), \QQ))\\
&= D(\QQ[x_1, y_1] \xrightarrow{\phi} \QQ[t] \xleftarrow{\psi} \QQ[x_2, y_2]),
\end{align*}
where $\deg~x_i=4$, $\deg~y_i=6$, $i=1, 2$, and $\deg~t=2$.
The morphisms $\Phi\colon\QQ[x_1, y_1]\otimes (\Lambda\langle c, \partial c\rangle, \partial)\to  \QQ[t]$ and $\psi\colon \QQ[x_2, y_2]\to  \QQ[t]$, where $(\Lambda\langle c, \partial c\rangle, \partial)$  is a contractible algebra as in Theorem~\ref{T:Surjectivization}, are defined as follows

\begin{align*}
\QQ[x_1, y_1]\otimes (\Lambda\langle c, \partial c\rangle, \partial)&\longrightarrow  \QQ[t] \longleftarrow \QQ[x_2, y_2]\\
x_1&\longmapsto -3t^2\longmapsfrom x_2\\
y_1&\longmapsto -2t^3\longmapsfrom y_2\\
c&\longmapsto t.
\end{align*}

One may see that for every element $[\xi, \eta]\in H^{\geq 1}(\tilde{D})$, we have that  $[\partial c, 0]\cdot [\xi, \eta]=0$. Thus $\depth H(\tilde{D})=0$. However,  $P=\{[\xi, 0]\mid \xi\in \QQ[x_1, y_1]\otimes \Lambda\langle c, \partial c\rangle\}$ is a prime ideal which is strictly contained in $\mathfrak{m}= H^{\geq 1}(\tilde{D})$. Therefore,   $\Kdim H(\tilde{D})\geq 1$.
\end{itemize}

\begin{rem}
Note that the argument above shows that unlike cohomogeneity one (smooth or topological) manifolds, cohomogeneity one Alexandrov spaces are not necessarily Cohen--Macaulay.
\end{rem}

We recall the following Lemma from \cite{GoMa14}:

\begin{lem}\cite[Lemma~2.5]{GoMa14}\label{L: Changing_the_rings}
Let $R$ and $S$ be two Noetherian graded $\prescript{*}{}{\text{local}}$ rings and let $\phi \colon R\to S$ be a homomorphism that makes $S$ into an $R$-module which is finitely generated. If $A$ is a finitely generated $S$-module, then we have

\[
\operatorname{depth}_R A = \operatorname{depth}_S A\quad \text{and} \quad\dim_R A = \dim_S A.
\]
In particular, $A$ is Cohen--Macaulay as an $R$-module if and only if it is Cohen--Macaulay as an $S$-module.
\end{lem}
\begin{lem}\label{L: Dim_of_Ring_Instead_of_Module}
The dimension of the $G$-equivariant cohomology $H^*_G(X)$ as an $H^*(\B G, \QQ)$-module is given by the Krull dimension of $H\big(D\big(H^*(\B K^{-}, \QQ) \to H^* (\B H, \QQ) \leftarrow H^* (\B K^{+}, \QQ)\big)\big)$, and its depth is the depth of the latter ring, i.e.~
\begin{align}\label{Eq:Dim_Depth_1}
\dim_{H^{*}(\B G, \QQ)} H_G^*(X)&=\dim H(\tilde D),
\end{align}

\begin{align}\label{Eq:Dim_Depth_2}
\operatorname{depth}_{H^*(\B G, \QQ)} H_G^*(X)&=\operatorname{depth} H(\tilde D).
\end{align}
\end{lem}

\begin{proof}
By Proposition~\ref{P:Equivariant_Cohomology}, there exists an isomorphism
\begin{align*}
\phi\colon H_G^*(X)\to H( \tilde{D}).
\end{align*}
Therefore, we have a homomorphism
\[
\psi=\phi H(A_{PL}(\pi))\colon H^*(\B G, \QQ)\to H(\tilde{D}).
\]
Since $H^*(\B G, \QQ)$ and $H(\tilde{D})$ are both Noetherian $\prescript{*}{}{\text{local}}$ \cite{Venkov, GoHaghMa}, and  $H(\tilde{D})$  is finitely generated $H^*(\B G, \QQ)$-module,  Lemma~\ref{L: Changing_the_rings} now yields the relations in  \eqref{Eq:Dim_Depth_1} and \eqref{Eq:Dim_Depth_2}.
\end{proof}



\section{Proof of the main theorems}\label{S:Proof}
In this section we prove our main results. We begin by recalling some algebraic notions and facts that we need later in the proofs of preliminary lemmata.

\begin{definition}\cite[Definition~5.3]{Kemper}
Let $A$ be an algebra over a field $k$. Then the \emph{transcendence
degree} of $A$ is defined as follows
\[
\trdeg(A) := \sup\{|T | \, \mid T \subseteq A \, \,\text{is finite and algebraically independen}t\}.
\]
\end{definition}

\begin{prop}\label{P:trdeg}\cite[Theorem~5.9]{Kemper}
Let $A$ be an affine
algebra over a field $k$. Then
\[
\Kdim(A) = \trdeg(A).
\]
\end{prop}

\begin{definition}
Let $A$ be a finitely generated non--negatively  graded algebra over a field $k$. By a \emph{homogeneous system of parameters} for $A$ we mean a sequence of homogeneous elements $F_{1},\ldots,F_{n}$ of positive degree in $A $ such that $n =\Kdim (A)$ and $A/\langle F_{1},\ldots,F_{n}\rangle$ has Krull dimension 0.
\end{definition}

\begin{definition}\cite[Definition~8.1]{Kemper}
Let $R$ be a ring and $S\subseteq R$ be a subring. An element $\alpha\in R$ is called \emph{integral} over $S$ if there exists a monic polynomial $p\in S[x]$ such that $p(\alpha)=0$.  $R$ is called integral over $S$ if all  of whose elements are integral over $S$. In this case, $R$ is called  an \emph{integral extension} of $S$.
\end{definition}

\begin{prop}\label{P:Integral_Extension_of_Polynomial_Ring}\cite[Theorem~1.5.17]{BrHe}
Let  $A$ be a non--negatively graded affine algebra over a field $k$. Set $n=\Kdim~A$. Then
\begin{itemize}
 \item[(1)] for  homogeneous elements $x_1,\ldots, x_n$  the following statements are equivalent:
\begin{itemize}
\item[(i)] $x_1,\ldots, x_n$ is a homogeneous system of parameters,
\item[(ii)] $A$ is an integral extension of $k[x_1,\ldots, x_n]$.
\end{itemize}
\item[(2)] For (1) there always exist homogeneous elements $x_1,\ldots, x_n$ satisfying (i) and hence (ii).
\end{itemize}
\end{prop}

\begin {lem}\label{L:rational_sphere_surjectivity}
Let $K, H$ be two compact  Lie groups whose classifying spaces $\B K$ and $\B H$ are Sullivan spaces and
let  $K/H$ be a connected positively curved homogeneous space.  Then the map
\[
H^*(B\iota) \colon H^*(\B K, \mathbb{Q})\to H^*(\B H, \mathbb{Q}),
\]
is surjective  if and only if $K/ H$ rationally is an odd--dimensional sphere,  where $\iota \colon H\to K$ is the inclusion map.
\end{lem}

\begin{proof}
Let $\iota_{0} \colon K_{0}\cap H\to K_{0}$. Then by Corollary~\ref{C_Sullivan_Model_of_BH}, $H^*(\B \iota)$ is surjective if and only if $H^*(\B \iota_{0})$ is surjective. Since $K/H\simeq K_{0}/K_{0}\cap H$, we can assume, without loss of generality, that $K$ is connected. Moreover, let $H^*(\B K, \mathbb{Q})=\QQ[x_{1}, \ldots, x_{k}]$ and $H^*(\B H, \mathbb{Q})=\QQ[y_{1}, \ldots, y_{l}]$.

First assume that  $K/H$ is rationally  an odd--dimensional sphere and let
$$(\QQ[y_{1}, \ldots, y_{l}]\otimes \Lambda \langle v_{1}, \ldots, v_{k}\rangle, d)
$$
be  a Sullivan model for $K/H$. Since $d\mid_{\QQ[y_{1}, \ldots, y_{l}]}=0$, every element in $\QQ[y_{1}, \ldots, y_{l}]$ has to be exact.
Using induction on the degree of the generators of $H^*(\B H, \mathbb{Q})$,  we show  that $H^{*}(\B \iota)$ is  surjective.  First  assume that for each $i$, $1\leq i\leq l-1$, and for each $j$, $1\leq j\leq k-1$,  we have that   $\deg y_{i}, \leq \deg y_{i+1}$ and $\deg v_{j}, \leq \deg v_{j+1}$. Then there exists an element of the form  $\sum a_{si}v_{si}\in \Lambda\langle v_{1}, \ldots, v_{k}\rangle$ such that $d(\sum a_{si}v_{si})=y_s$ for all $s$ with $\deg y_{s}=\deg y_{1}$. Therefore, $H^{*}(\B \iota)(\sum a_{si}v_{si})=y_s$. Now suppose that for each $i$ with $\deg y_{1}\leq \deg y_{i}\leq t<\deg y_{l}$, there exists  $z_i\in \QQ[x_{1}, \ldots, x_{k}]$ such that $H^{*}(\B\iota)(z_i)=y_i$. To finish the proof, we  show that each $y_{i}$ with $\deg y_{i}=t+1$ lies in the image of $H^{*}(\B\iota)$. Since $y_{i}$ is an exact element in $(\QQ[y_{1}, \ldots, y_{l}]\otimes \Lambda \langle v_{1}, \ldots, v_{k}\rangle, d)$, there exists an element  $\sum b_{ij} v_{ij}+P(y_1, \ldots, y_{r_1}, v_1, \ldots, v_{r_2})$ such that
  \begin{align*}
  y_{i}&=d(\sum b_{ij}v_{ij}+P(y_1, \ldots, y_{r_1}, v_1, \ldots, v_{r_2}))\\
  &=\sum b_{ij}dv_{ij}+d(P(y_1, \ldots, y_{r_1}, v_1, \ldots, v_{r_2}))\\
  &=\sum b_{ij}H^{*}(\B\iota)(x_{ij})+d(P(y_1, \ldots, y_{r_1}, v_1, \ldots, v_{r_2})),
  \end{align*}
  where $P(y_1, \ldots, y_{r_1}, v_1, \ldots, v_{r_2})$ is a polynomial with $\deg y_j<\deg y_{i} $ and $\deg v_j<\deg y_{i} $, for $1\leq j\leq r_1$ and $1\leq j\leq r_2$, respectively. Since  $d(P(y_1, \ldots, y_{r_1}, v_1, \ldots, v_{r_2}))$ is a polynomial in $\QQ[y_{1}, \ldots, y_{l}]$
 which is generated by the algebra generators whose degrees are smaller than  $\deg y_{i}$, by the induction hypothesis, there exists $z\in H^*(\B K, \mathbb{Q})$ such that $H^{*}(\B\iota)(z)=d(P(y_1, \ldots, y_{r_1}, v_1, \ldots, v_{r_2})$. Consequently, $y_{i}\in \im H^{*}(\B\iota)$

\vspace{3mm}
Conversely, assume that $H^*(\B\iota)$ is surjective and $K/H$ is not rationally an odd-dimensional sphere. We show that the other possibilities for $K/H$ yields a contradiction.  Note that since an odd-dimensional positively curved homogeneous space whose universal covering space is a sphere is rationally nilpotent by Proposition~\ref{P:non-simply-connected_homogeneous_space_Rationally_Nilpotent}, it has the same rational homotopy type as an odd-dimensional sphere by Proposition~\ref{L_Sullivan_Model_of_Non_Simply_Connected}.    First,  we  show that $K/H$ is not an even-dimensional positively curved homogeneous space. If so, then by Lemma \ref{L:Equal_rank_injectivity} we have that $H^*(\B\iota_{0})$ is injective which implies that $H^{*}(K/H, \QQ)=\QQ$. This can only occur  if $K/H$ is not simply-connected (for example,  $K/H=\mathbb{RP}^{2n}$).  However,  Proposition \ref{L:Even_Dim_Sullivan} rules out this case.  Now assume that $K/H$ is odd-dimensional. We rule out the cases where the universal covering of $K/H$ is $W^7_{p,q}$ or $B^{13}$. First note that by Proposition~\ref{P:non-simply-connected_homogeneous_space_Rationally_Nilpotent}, $K/H$ is rationally nilpotent and hence a Sullivan space by Proposition~\ref{P:Rationally_Nilpotent_is_Sullivan}. As a result, Proposition~\ref{L_Sullivan_Model_of_Non_Simply_Connected} implies that $K/H$ has the same rational homotopy type as $W^7_{p,q}$ and $B^{13}$, respectively. Moreover, it is known that $W^7_{p,q}\simeq_{\mathbb{Q}} S^2\times S^5$, and $B^{13}\simeq_{\mathbb{Q}} \mathbb{CP}^2\times S^9$. In particular,  $H^{1}(K/H, \QQ)= 0$ and $H^{2}(K/H, \QQ)\neq 0$. Let $(H^{*}(\B H, \QQ)\otimes \Lambda V_{K}, d)$ be a Sullivan model for $K/H$.  Since $H^{2}(K/H, \QQ)\neq 0$, there  exists a cocycle  $x$ of degree $2$ in $H^{*}(\B H, \QQ)\otimes \Lambda V_{K}$ which is not exact. Since $H^{1}(K/H, \QQ)= 0$ we have that $x\in H^{*}(\B H, \QQ)$. Since  $\deg x=2$ and $H^*(\B\iota)$ is surjective, then $x=H^*(\B\iota)(\sum a_{i}x_{i})$, where  $a_{i}\in \QQ$ and $x_{i}\in H^*(\B K, \mathbb{Q})=\QQ[x_{1}, \ldots, x_{k}]$ with $\deg x_{i}=2$. It implies, by the definition of differential of $(H^{*}(\B H, \QQ)\otimes \Lambda V_{K}, d)$, that $x$ is exact. Therefore,  $K/H$, up to the universal cover, cannot be  $W^7_{p,q}$ or $B^{13}$.
\end{proof}

\begin{lem}\label{L:Equal_rank_injectivity}
Let $K, H$ be two compact  Lie groups  with $\rank H=\rank K$.
 Assume additionally  that the classifying spaces $\B K$ and $\B H$ are Sullivan spaces. Then the morphism

\[
H^{*}(\B\iota) \colon H^*(\B K, \mathbb{Q})\to H^*(\B H, \mathbb{Q}),
\]
is injective.
\end{lem}
\begin{proof}
As in the proof of  Lemma~\ref{L:rational_sphere_surjectivity},  $H^{*}(\B\iota)$   is injective if and only if $H^{*}(\B\iota_{0})$ is injective, where $\iota_{0} \colon  H_0\to K_{0}$.  Hence, we can assume, without loss of generality, that $K$ and $H$ are connected.

Denote by $T$ a common maximal torus. Then we identify the Lie group cohomology with the invariants of the cohomology of the maximal torus under the respective Weyl group action, i.e.~$H^*(\B H, \QQ)=H^*(\B T, \QQ)^{W(H)}$ and $H^*(\B K, \QQ)=H^*(\B T, \QQ)^{W(K)}$. The map $H^{*}(\B\iota)$ is the map induced by the inclusion $H\hookrightarrow{} K$, i.e.~the morphism $H^*(\B T, \QQ)^{W(K)}\to H^*(\B T,\QQ)^{W(H)}$ induced by the identity on the maximal torus, which then is necessarily injective, since $W(H)=N_H(T)/T \subseteq N_K(T)/T=W(K)$ using the description via normalizers.
\end{proof}

Since the Euler characteristics of  even dimensional  positively  curved homogeneous spaces $K/H$ are positive, we have that  $\rank K=\rank H$ in this case.  Hence  the following corollary is immediate.  

\begin{cor}\label{L:Equal_rank_injectivity_PC}
Let $K, H$ be two compact  Lie groups  
such that $K/H$ is homeomorphic to an even-dimensional positively curved homogeneous space. 
 Assume additionally  that the classifying spaces $\B K$ and $\B H$ are Sullivan spaces. Then the map

\[
H^{*}(\B\iota) \colon H^*(\B K, \mathbb{Q})\to H^*(\B H, \mathbb{Q}),
\]
is injective.
\end{cor}

\begin{lem}\label{L:Zero-divisor_in_even_degree_1}(cf. \cite[Lemma~B]{Eakin})
Let $K$ be a compact Lie group and  $H$ be a closed subgroup
and let  $\iota\colon H\hookrightarrow K$ be the inclusion map. Suppose that  the map $H^{*}(\B\iota)$ is injective. Then $\rank~H=\rank~K$.
\end{lem}

\begin{proof}
Since $H$ is a subgroup of $K$, we only need to prove that $\rank~K\leq \rank~H$. First note that by \cite[Corollary~2.7]{GoMa14}, we have that  $H^*(\B K, \mathbb{Q})$ and $H^*(\B H, \mathbb{Q})$ are Cohen--Macaulay rings whose Krull dimensions are equal to their ranks. Further, since $H^{*}(\B\iota)$ is injective, $H^*(\B K, \mathbb{Q})$ is isomorphic to a graded subalgebra of  $H^*(\B H, \mathbb{Q})$.  Therefore, by Proposition~\ref{P:trdeg},  we have that
\begin{align*}
\rank~K=\Kdim~H^*(\B K, \mathbb{Q})=\trdeg~H^*(\B K, \mathbb{Q}) \\
\leq\trdeg~H^*(\B H, \mathbb{Q})&=\Kdim~H^*(\B H, \mathbb{Q})=\rank~H.
\end{align*}
\end{proof}

\begin{lem}\label{L:Zero-divisor_in_even_degree_2}
Let $X$ be a generalized cohomogeneity one space given by  a group diagram $(G, H, K^{-}, K^{+})$ with $\B H$, $\B K^{-}$, and $\B K^{+}$ Sullivan spaces. Assume further that $\rank~H<\max\{\rank~K^{-}, \rank K^{+}\}$.
 If for the morphisms
\[
\phi=H^{*}(\B\iota^-) \colon H^*(\B K^-, \mathbb{Q})\to H^*(\B H, \mathbb{Q}),
\]
and
\[
\psi=H^{*}(\B\iota^+)\colon H^*(\B K^+, \mathbb{Q})\to H^*(\B H, \mathbb{Q}),
\]
the following relation holds
$$\im \phi +\im \psi\subsetneq H^*(\B H, \mathbb{Q}),$$
then $H(\tilde{D})$ has a zero-divisor of even degree which is not a zero divisor in $H^{\textrm{even}}(\tilde{D})$.
\end{lem}

\begin{proof}
First we show that there exists a nonzero element $[(\beta, 0)]$ of odd degree in $H(\tilde{D})$. By assumption,  let  $z\in H^{*}(\B H, \QQ)\setminus (\im \phi+ \im\psi)$. Let $\Phi$ be a surjective extension of $\phi$ as in Remark~\ref{R:Surjective_Trick}, defined on $H^*(\B K^-, \mathbb{Q})\otimes C$, where $C=\Lambda (V\oplus dV)$  is a contractible algebra and $V$ is concentrated in even degrees. By the construction of $\Phi$ (see \cite[Page~148]{RHT}), we can choose $\Phi$ and $V$ such that, without restriction, there is some $w\in V$ with the property that $\Phi(w)=z$. We show that $[dw, 0]\neq 0$. By contradiction, we suppose that $(dw, 0)=d(\xi, x)$, where $\Phi(\xi)=\psi(x)$.  Then $dw=d\xi$. There exist basis elements $x_1, \ldots, x_n$ of the free $\mathbb{Q}$-module $H(\B K^-, \mathbb{Q})$  and $c_i\in C$, $1\leq i\leq n$, such that $\xi=\sum_{i=1}^{n}  x_i\otimes c_i$. We have
\[
0\neq dw=d\xi=\sum_{i=1}^{n} x_i\otimes  dc_i.
\]
Therefore,  there is $J\subseteq \{1, \ldots, n\}$, such that $x_{j}=1$ for $j\in J$ and $dw=\sum_{j\in J}dc_{j}$, and $dc_{\gamma}=0$ for $\gamma\in\Gamma= \{1, \ldots, n\}\setminus J$ ($\Gamma$ might be an empty set).  Since $C$ is a contractible algebra, it follows that  either $c_{\gamma}=d c'_{\gamma}$ for some $c'_{\gamma}\in C$ or $c_{\gamma}\in \QQ$, for $\gamma\in \Gamma$.
Now write each $c_{i}\in \Lambda (V\oplus dV)$ in the form of $q_{i}+v_{i}+dv'_{i}+\eta_{i}$, where $q_{i}\in \QQ$, $v_{i}, v'_{i}\in V$, and $\eta_{i}\in \Lambda^{\geq 2}(V\oplus dV)$.  We have that
$$d w=\sum_{j\in J} (dv_{j}+d\eta_{j}).$$
Since $d w\in dV$ and $d\eta_{j}\in \Lambda^{\geq 2}(V\oplus dV)$, we conclude that $d\eta_{j}=0$, and, since $C$ is a contractible algebra, $\eta_{j}=d\eta'_{j}$.  Therefore,
$$d w=\sum_{j\in J} dv_{j}=d(\sum_{j\in J} v_{j}).$$  As a result, $w=\sum_{j\in J} v_{j}$ since $d\colon V\to dV$ is a bijection. Now we rewrite $\xi$ as follows:
\begin{align*}
\xi&=\sum_{j\in J} c_{j}+\sum_{\gamma'\in \Gamma'}  x_{\gamma'}\otimes 1+\sum_{\gamma\in \Gamma}  x_{\gamma}\otimes dc'_{\gamma}\\
&=w+\sum_{j\in J} d\eta'_{j}+\sum_{\gamma'\in \Gamma'}  x_{\gamma'}\otimes 1+\sum_{\gamma\in \Gamma}  x_{\gamma}\otimes dc'_{\gamma}.
\end{align*}
We obtain
\begin{align*}
\psi(x)&=\Phi(\xi)\\
&=\Phi(w+\sum_{j\in J} d\eta'_{j}+\sum_{\gamma'\in \Gamma'}  x_{\gamma'}\otimes 1+\sum_{\gamma\in \Gamma}  x_{\gamma}\otimes dc'_{\gamma})\\
&=z+\phi (\sum_{\gamma'\in \Gamma'}  x_{\gamma'})+0.
\end{align*}
Consequently, $$z=\psi(x)+\phi(-\sum_{\gamma'\in \Gamma'}  x_{\gamma'})\in \im\phi+\im\psi,$$
which yields a contradiction by the choice of $z$.

 By assumption on the rank and without restriction,  we have two cases: $\rank~H=\rank K^{-}$ and $\rank H<\rank K^{+}$, or $\rank H<\rank K^{-}\leq\rank K^{+}$. First assume that  $\rank~H=\rank K^{-}$ and $\rank H<\rank K^{+}$. Then by Lemma~\ref{L:Zero-divisor_in_even_degree_1} there exists $0\neq\alpha\in H^{*}(\B K^{+}, \mathbb{Q})$ such that $\psi(\alpha)=0$. Therefore, $0\neq[(0, \alpha)]\in H(\tilde{D})$.  Observe that $[(0, \alpha)]$ is a zero divisor in $H(\tilde{D})$, for $[(0, \alpha)].[(dw, 0)]=0$. We show that $[(0, \alpha)]$ is not a zero-divisor in $H^{\textrm{even}}(\tilde{D})$. Let there exist $[(a,b)]\in H^{\textrm{even}}(\tilde{D})$ such that $(a, b)(0,\alpha)=d(x, y)=(dx, 0)$. This implies that $b\alpha=0$. Since $H(\B K^{+}, \mathbb{Q})$ is a polynomial
algebra, it does not have a zero-divisor. Thus $b=0$.  Therefore, $\Phi(a)=0$. If $a\in H^{*}(\B K^{-}, \mathbb{Q})$, since by Lemma~\ref{L:Equal_rank_injectivity},  $\phi$ is injective, $a=0$. Assume now that $a\in H(\B K^{-}, \mathbb{Q})\otimes C$. Then $a=\sum x_{i}\otimes c_{i}$, where $x_{i}$'s are elements of a basis of the $\QQ$-module $H^{*}(\B K^{-}, \mathbb{Q})$ and $c_{i}$'s are of even degrees. Since $(a, 0)$ is a cocycle, we have $\sum x_{i}\otimes dc_{i}=0$. Since $x_{i}$'s are basis elements, $dc_{i}=0$. Then either $c_{i}\in \QQ$ or $c_{i}=d c'_{i}$ as $C$ is a contractible algebra, where $c'_{i}$'s have odd degrees, whence,  $a=\sum x_{i}\otimes 1+\sum x_{j}\otimes dc'_{j}$. As a result,
$0=\Phi(a)=\phi(\sum x_{i}\otimes 1)+0$ which implies that $\sum x_{i}\otimes 1=0$, for $\phi$ is injective. Moreover, because $c'_{i}$'s have odd degrees,  it gives that $(\sum x_{j}\otimes c'_{j}, 0)\in \tilde{D}$ and thus $(a, 0)=d(\sum x_{j}\otimes c'_{j}, 0)$.

Now let $\rank H<\rank K^{-}\leq\rank K^{+}$.  Then  by Lemma~\ref{L:Zero-divisor_in_even_degree_1}, there exist $0\neq\alpha_{1}\in H^*(\B K^-, \mathbb{Q})$ and $0 \neq\alpha_{2}\in H^*(\B K^+, \mathbb{Q})$ such that
$0\neq[(\alpha_{1}, 0)], 0\neq[(0, \alpha_{2})]\in H(\tilde{D})$. We claim that $[(\alpha_{1}, \alpha_{2})]$ is a zero divisor in $H(\tilde{D})$ but not in $H^{\textrm{even}}(\tilde{D})$.  Note that $[(\alpha_{1}, \alpha_{2})].[(dw, 0)]=0$, for  by $(\alpha_{1}w, 0)$ is well-defined and hence $(\alpha_{1}dw, 0)$ is exact. Now we show that $[(\alpha_{1}, \alpha_{2})]$ is not a zero divisor in $H^{\textrm{even}}(\tilde{D})$. Assume  that there exists $[(a, b)]\in H^{\textrm{even}}(\tilde{D})$ such that $(a, b)(\alpha_{1},\alpha_{2})=d(x, y)=(dx, 0)$.  From $b\alpha_{2}=0$ we deduce that $b=0$.

If $a\notin H^{*}(\B K^-, \QQ)$, but $a\in H^{*}(\B K^-, \QQ)\otimes C$, then as before, $[(a, 0)]=0$. Let $a\in H^{*}(\B K^-, \QQ)$. Thus $a\alpha_{1}\in H^{*}(\B K^-, \QQ)$.  Since $a\alpha_{1}=dx$, we conclude that $a\alpha_{1}=0$ and hence $a=0$.
\end{proof}

\begin{lem}\label{L:Reg_Seq_H_is_Reg_Seq_Heven}
Suppose that  $R=H(\tilde{D})$ and $S=H^{\textrm{even}}(\tilde{D})$.
Then
\begin{itemize}
\item[(1)] Every homogeneous regular sequence of $R$ is a regular sequence of $S$.

\item[(2)] If $R$ is Cohen--Macaulay, so is $S$.
\end{itemize}
\end{lem}

\begin{proof}
\begin{itemize}
\item[(1)]
Let $(x_1, \ldots, x_n)$ be a homogeneous regular sequence in $R$. First note that since odd-degree elements are nilpotent, $\deg~x_i$ is even, for $i=1, \ldots, n$. Therefore, $(x_1, \ldots, x_n)\in S$. Let
$$x_i+\langle x_1,\ldots, x_{i-1}\rangle_S\in \frac{S}{\langle x_1,\ldots, x_{i-1}\rangle_S}, $$
where $\langle x_1,\ldots, x_{i-1}\rangle_S$ means an ideal generated by $x_1,\ldots, x_{i-1}$ in $S$.
If  there exists $a\in S$ such that
$$x_ia\in \langle x_1,\ldots, x_{i-1}\rangle_S\subseteq  \langle x_1,\ldots, x_{i-1}\rangle_R,$$
then since $(x_1, \ldots, x_n)$ is a  regular sequence in $R$, we have that $a\in \langle x_1,\ldots, x_{i-1}\rangle_R$. Let
 $a=\sum \lambda_jx_j$,
for $\lambda_j\in R$. Then
$$a=\sum_{\deg \lambda_{j}\,\text{odd}} \lambda^{odd}_jx_j+\sum_{\deg \lambda_{j}\,\textrm{even}}\lambda^{\textrm{even}}_jx_j.$$
Since $\deg~a$ is even, we have that $\sum \lambda^{odd}_jx_j=0$. This  implies that
 $$a=\sum_{\deg \lambda_{j}\,\textrm{even}}\lambda^{\textrm{even}}_jx_j\in \langle x_1,\ldots, x_{i-1}\rangle_S.$$
 Consequently, $(x_1, \ldots, x_n)$ is a regular sequence in $S$.

 \item[(2)] Let $x_1, \ldots, x_n$ be a homogeneous system of parameters, which by Proposition~\ref{P:Integral_Extension_of_Polynomial_Ring} always exists. Since $R$ is Cohen--Macaulay, every homogeneous system of parameters is a regular sequence. Whence $(x_1, \ldots, x_n)$ is a maximal regular sequence in   $\mathfrak{m}_{R}=\sum_{i\geq 1} R^{i}$.
 By Part (1), we have that $(x_1, \ldots, x_n)$ is a regular sequence in $\mathfrak{m}_{S}$. Thus $n\leq \depth~ S\leq \Kdim~S$.  To show that $\Kdim~S\leq n$, notice that $S$ and $R$ are finitely generated $\QQ$-algebras. Therefore, by Proposition~\ref{P:trdeg},
 $$\Kdim~S=\trdeg~S\leq \trdeg~R =\Kdim~R=n.$$
  \end{itemize}
\end{proof}

\begin{cor}\label{C:Integral_Extension}
Suppose that  $R=H(\tilde{D})$ and $S=H^{\textrm{even}}(\tilde{D})$.
Then $R$ is an integral extension of $S$ if $R$ is Cohen--Macaulay.
\end{cor}

\begin{proof}
Let $(x_1, \ldots, x_n)$ be a maximal regular sequence in $R$.  Hence by Lemma~\ref{L:Reg_Seq_H_is_Reg_Seq_Heven}, Part(2), it is a maximal regular sequence in $S$ as well. Then by Proposition~\ref{P:Integral_Extension_of_Polynomial_Ring}, $R$ and $S$ are integral extensions of $\QQ[x_{1},\ldots, x_{n}]$. Since $$\QQ[x_{1},\ldots, x_{n}]\subseteq S\subseteq R,$$
$R$ is in particular an integral extension of $S$.
\end{proof}

\begin{lem}\label{L:Reg_Seq_Heven_is_Reg_Seq_H}
Let $S$ be a subring of $R$ with $\Kdim~S=\Kdim~R=n$. Assume further that $R$ and $S$ are both Cohen--Macaulay and $R$ is an integral extension of $S$. Then any maximal homogeneous regular sequence of $S$ is a regular sequence of $R$.
\end{lem}

\begin{proof}
Let $(x_1, \ldots, x_n)$ be a maximal homogeneous regular sequence of $S$. Then by \cite[Corollary~A.8, Page~415]{BrHe}, we have

\begin{align*}
\Kdim~\frac{R}{\langle x_1,\ldots, x_{n}\rangle_R}&=\Kdim~\frac{S}{\langle x_1,\ldots, x_{n}\rangle_R\cap S}\\
&=\Kdim~\frac{S}{\langle x_1,\ldots, x_{n}\rangle_S}=0.
\end{align*}
As a result, $(x_1, \ldots, x_n)$ is a system of parameters in $R$ and since $R$ is Cohen--Macaulay, $(x_1, \ldots, x_n)$ is a regular sequence in $R$.
\end{proof}

Now we  prove Theorem~\ref{T:Main_THM_A}.

\begin{proof}[Proof of Theorem~\ref{T:Main_THM_A}]
We first prove the ``if'' part.  We should mention that the idea for the proof of this part is basically a modification to the setting of generalized cohomogeneity one spaces of previous work in \cite{GoRo} and \cite{GoMa14}.
Let $\rank H=\rank K^-=\rank K^+$. Then one can use \cite[Corollary~4.3]{GoRo} (the proof passes verbatim to our setting) to show that $X$ is a Cohen--Macaulay space.

Now assume the situation of  Part (2).
 Let $r= \max \{\rank K^-, \rank K^+\}=\rank~K^{-}$, without restriction.  Then $\rank~H=r-1$ by assumption.  Two cases may occur: either $\rank~H=\rank~K^+$, or $\rank~H<\rank~K^+$. If $\rank~H=\rank~K^+$, then by Lemma~\ref{L:Equal_rank_injectivity},  $H^*(\B\iota^{+})$
  is injective and $K^+/H$ has rational cohomology concentrated in even degrees (see Proposition \ref{P:Cohomology_Homogeneous_Sp}). 
 Hence, a similar argument as in the proof of Theorem  1.1 in  \cite{GoMa14} shows that $X$ is Cohen--Macaulay. If $\rank~H<\rank~K^+$,  then  $\rank K^-=\rank K^+$.  Again a similar argument as in the proof of Theorem  1.1 in \cite{GoMa14} gives the result.

Now we prove the ``only if'' part. Let the equivariant cohomology of  $X$ be Cohen--Macaulay and $\rank H< \max \{\rank K^-, \rank K^+\}$. By contradiction, assume that
  $$\im H^{*}(\B\iota^-)+\im H^{*}(\B\iota^+)\subsetneq H^*(\B H, \mathbb{Q}).$$
 Then by Lemma \ref{L:Zero-divisor_in_even_degree_2}, $H(\tilde{D})$ has a zero divisor $\alpha$  of even degree which is not a zero divisor in $H^{\textrm{even}}(\tilde{D})$. We extend $\alpha$ to a maximal regular sequence in  $H^{\textrm{even}}(\tilde{D})$. However, since $X$ is Cohen--Macaulay,  by Lemma~\ref{L:Reg_Seq_Heven_is_Reg_Seq_H}, this regular sequence  is a regular sequence of  $H(\tilde{D})$ as well. This yields a contradiction since $\alpha$ is a zero divisor in $H(\tilde{D})$.
\end{proof}

Now we proceed  with the proof of Theorem~\ref{T:Main_THM_B}. First we need the following proposition

\begin{prop}\label{P:Join}
Let $G_1/H_1\ast G_2/H_2$ be a join of PCHSs of compact Lie groups such that $G_1$ and $G_2$ are connected and $\B H_1$, $\B H_2$ are Sullivan spaces. Suppose that neither $H^*(\B G_1)\to H^*(\B H_1)$ nor $H^*(\B G_2)\to H^*(\B H_2)$ is surjective. Then $H^*_{G_1\times G_2}(G_1/H_1\ast G_2/H_2,\QQ)$ is not Cohen--Macaulay.
\end{prop}
\begin{proof}
Denote by $\QQ[t_1,\ldots, t_k]$ a Sullivan model of $H^*(\B H_1,\QQ)$, by $\QQ[s_1,\ldots, s_l]$ a Sullivan model of  $H^*(\B H_2,\QQ)$, by $\QQ[x_1,\ldots, x_n]$ a Sullivan model of $H^*(\B G_1,\QQ)$, by $\QQ[y_1,\ldots, y_m]$ a Sullivan model of $H^*(\B G_2,\QQ)$. It follows that the relevant algebraic double mapping cylinder $\tilde D$ is the one of
\begin{align*}
\QQ[x_1,\ldots, x_n,s_1,\ldots, s_l]\xrightarrow{\phi^-} \QQ[t_1,\ldots, t_k,s_1,\ldots, s_l] \xleftarrow{\phi^+} \QQ[t_1,\ldots, t_k,y_1,\ldots, y_m]
\end{align*}
with the obvious induced morphisms and identities respecting the product structures.

In view of Theorem \ref{T:Main_THM_A} we show that $\phi^-+\phi^+$ is not surjective. Since neither side is surjective, we observe that there is some generators $s_\alpha$ and  $t_\beta$ such that $t_\beta\not \in \phi^-(\QQ[x_1,\ldots, x_n])$ and $s_\alpha\not \in \phi^+(\QQ[y_1,\ldots, y_m])$. Let $s_\alpha$, $t_\beta$ be of minimal degree with this property.

We aim to prove that $s_\alpha t_\beta\not \in \im \phi^- + \im \phi^+$. Assume the contrary, namely that there is an element $(z^-,z^+)\in \QQ[x_1,\ldots, x_n,s_1,\ldots, s_l] \oplus \QQ[t_1,\ldots, t_k,y_1,\ldots, y_m]$ with $\phi^-(z^-) + \phi^+(z^+)=s_\alpha t_\beta$. By definition of $\phi^\pm$ we have that $\phi^-(x_j)$ is a polynomial in the $t_i$ for $1\leq i\leq n$, (and $\phi^+(y_j)$ is a polynomial in the $s_i$ for $1\leq i\leq m$). We define the word-length of a polynomial in $\QQ[t_1,\ldots, t_k,s_1,\ldots, s_l]$ (respectively in $\QQ[t_1,\ldots, t_k]$, $\QQ[s_1,\ldots, s_l]$)  as the minimum word-length of any of its non-trivial monomials--we cancel monomials as far as possible. Due to multiplicativity, without restriction, $\phi^-(z^-)$ has word-length at most two (in the $t_i$, $s_i$) and actually word-length one in the $t_i$. Since $\phi^-$ is the identity on the $s_i$, $1 \leq i\leq m$, and since $\phi^-(x_i)\in \QQ[t_1,\ldots, t_n]$, we derive that this is only possible if there is, without restriction, an $x_i$ with $\phi^-(x_i)=t_\alpha + \gamma$ (with $\gamma\in \QQ[t_1,\ldots, t_n]$ of word-length at least $2$ in the $t_i$).
As $t_\alpha$ was chosen of minimal degree, i.e.~since $\phi^-$ surjects onto all $t_i$ smaller than $\deg t_\alpha$, it follows that $\gamma$ lies in $\im (\phi^-)$ (since $\phi^-$ is a morphism of rings). We deduce that also $t_\alpha\in \im (\phi^-)$ contradicting our original assumption.
\end{proof}

\begin{proof} [Proof of Theorem~\ref{T:Main_THM_B}]
Let $F^+, F^-$  be as in Theorem~\ref{T:Main_THM_B}, and assume that there exists a cohomogeneity one Alexandrov space with group diagram $(G, H, K^{-}, K^{+})$ such that the classifying spaces $\B H$, $\B K^\pm$ are  Sullivan spaces  and  the action on $X$  is not Cohen--Macaulay. By contradiction, suppose that one of the following cases occurs:
\begin{itemize}
\item  [(i)]
$\rank K^-=\rank H=\rank K^+$, or
\item [(ii)] $K^{-}/H\notin \{W^{7}_{p, q}/\Gamma, B^{13}, M^\textrm{even}\}$, or
\item [(iii)] $K^{+}/H\notin \{W^{7}_{p, q}/\Gamma, B^{13}, M^\textrm{even}\}$.
\end{itemize}
By Theorem~\ref{T:Main_THM_A}, we immediately rule out Item (i). Therefore, $\rank H<\max\{\rank K^{-}, \rank K^+\}$. If Item (ii) happens, then we deduce that, by the classification of positively curved homogeneous spaces and by Proposition~\ref{P:non-simply-connected_homogeneous_space_Rationally_Nilpotent}, $K^{-}/H$ has to be rationally an odd--dimensional sphere. By Lemma~\ref{L:rational_sphere_surjectivity}, the morphism $H^{*}(\B \iota^{-})$ is surjective and therefore, $\im H^{*}(\B\iota^-)+\im H^{*}(\B\iota^+)=H^*(\B H, \mathbb{Q})
$. Again it follows from Theorem~\ref{T:Main_THM_A} that the action is Cohen--Macaulay contradicting our assumption. A similar argument shows that Item (iii) cannot occur either.

Now let $F^-, F^{+}$ be  positively curved homogeneous spaces   such that  they  can be written as quotients of compact Lie groups whose classifying  spaces are Sullivan spaces. We show that we can construct a cohomogeneity one Alexandrov space with group diagram $(G, H, K^{-}, K^{+})$ and with classifying spaces $\B H$, $\B K^\pm$  being Sullivan spaces such that the action on $X$ is not Cohen--Macaulay if both the following hold:
\begin{itemize}
\item
At most one of the spaces $F^+, F^-$ is even dimensional, and
\item
$F^{\pm}/H\in \{W^{7}_{p, q}/\Gamma, B^{13}, M^\textrm{even}\}$.
\end{itemize}
To this end,  from the list of positively curved homogeneous spaces in \cite{WZ},  we choose compact Lie groups $G_{i}, H_{i}$, $i=1, 2$ such that $F^-=G_{1}/H_{1}$ and $F^+=G_{2}/H_{2}$ satisfying the conditions of the assertion. Let $X=G_{1}/H_{1}\ast G_{2}/H_{2}$,  the spherical join of $G_{1}/H_{1}$ and  $G_{2}/H_{2}$, which by Proposition~\ref{P:Join_Action} is a cohomogeneiy one Alexandrov space with  group diagram
$$(G_{1}\times G_{2}, H_{1}\times H_{2}, G_{1}\times H_{2}, H_{1}\times G_{2}). $$
By the choice of $F^{\pm}$, Lemma~\ref{L:rational_sphere_surjectivity} implies that none of the morphisms $H^*(\B G_1)\to H^*(\B H_1)$ and  $H^*(\B G_2)\to H^*(\B H_2)$ is surjective. Therefore, by Proposition~\ref{P:Join}  the join action on $X$ is not Cohen--Macaulay.
\end{proof}

We conclude this section by the proof of Theorem~\ref{T:Cohom_1_Orbifolds}.

\begin{proof}[Proof of Theorem~\ref{T:Cohom_1_Orbifolds}]
Let $X$ be a closed simply-connected  smooth orbifold and $G$ be a compact connected Lie group
which acts on   $X$ by cohomogeneity one with  a group diagram $(G, H, K^-, K^+)$, where the classifying spaces of the isotropy groups  $H$, $K^-$, and $K^+$ are Sullivan  spaces. By the structure theorem for cohomogeneity one actions on smooth orbifolds \cite{Gonzalez}, the singular normal fibers $K^{\pm}/H$ are diffeomorphic to spherical space forms. Since $K^{\pm}/H$ are in particular positively curved homogeneous spaces, $X$ is equivariantly homeomorphic to a cohomogeneity one Alexandrov space (with the same group diagram). If $\rank K^-=\rank H=\rank K^+$, then the action is Cohen--Macaulay by Theorem~\ref{T:Main_THM_A}. Let $\rank H<\max\{\rank K^{-}, \rank K^+\}$. Hence, at least on the spaces $K^{\pm}/H$ is an odd--dimensional spherical space form which by
Proposition~\ref{P:non-simply-connected_homogeneous_space_Rationally_Nilpotent}, it is rationally an  odd--dimensional sphere. Therefore, by Lemma~\ref{L:rational_sphere_surjectivity}, at least one of the morphisms
\begin{align*}
H^*(\B\iota^-) \colon H^*(\B K^-, \mathbb{Q})\to H^*(\B H, \mathbb{Q}),
\end{align*}
or
\begin{align*}
H^*(\B\iota^+)\colon H^*(\B K^+, \mathbb{Q})\to H^*(\B H, \mathbb{Q}),
\end{align*}
is surjective and in particular $\im H^{*}(\B\iota^-)+\im H^{*}(\B\iota^+)=H^*(\B H, \mathbb{Q})
$. Theorem~\ref{T:Main_THM_A} now finishes the proof.
\end{proof}


\section{On relaxing conditions}\label{S:PC}

This section is devoted to discussing which of the prerequisites of Theorem \ref{T:Main_THM_A} are necessary or can be relaxed. We shall do so by a sequence of examples which show that
\begin{itemize}
\item the condition of classifying spaces being Sullivan might be relaxed to a certain degree.
\item the result becomes wrong when passing to corank $2$ or higher.
\item there are generalized cohomogeneity one spaces for which the sum of the maps induced in cohomology is surjective whereas no single map $H^*(\B \iota^\pm)$ is.
\end{itemize}

\begin{example}[Sullivan spaces]\label{E:Sullivan_Condition}
There might be a certain chance to generalize the characterization from Theorem \ref{T:Main_THM_A} beyond the technical assumptions of ``Sullivan spaces" as the following example shows:  Here, one fibre will be $\mathbb{RP}^2$, which cannot be written as a quotient of compact Lie groups $K^-$, $H$ with $\B K$ and $\B H$ being Sullivan spaces (see Proposition \ref{L:Even_Dim_Sullivan}).

So consider the cohomogeneity one Alexandrov space  $X=\mathbb{RP}^2\ast W^7$ with the group diagram $(S^3\times \SU(3), N_{S^3}(S^1)\times S^1, S^3\times S^1, N_{S^3}(S^1)\times \SU(3))$, where $N_{S^3}(S^1)$ is the normaliser of $S^{1}$ in $S^{3}$. Note that the normal fibers are
$$\frac{S^3\times S^1}{N_{S^3}(S^1)\times S^1}=\mathbb{RP}^2, \qquad \frac{N_{S^3}(S^1)\times \SU(3)}{N_{S^3}(S^1)\times S^1}=W^7.$$
\\

 We claim that the equivariant cohomology of $X$ is Cohen--Macaulay even though this is not implied by Theorem~\ref{T:Main_THM_A}. This is not surprising however since  $\B N_{S^3}(S^1)$ is not a Sullivan space, for $H^*(\B N_{S^3}(S^1), \QQ)=\QQ[t]$ with $\deg t=4$, while $H^*(\B S^1, \QQ)=\QQ[s]$ with $\deg s=2$ (see for example~\cite[Theorem~15.9 and Problem~15.9]{MiSt} and cf.  Proposition~\ref{L_Sullivan_Model_of_Non_Simply_Connected}). To see why $X$ is a Cohen--Macaulay space, we slightly modify the proof of  Proposition~\ref{P:Equivariant_Cohomology} to get the following model for $X_G$  
 \[
 D((\QQ[u, s'], 0)\to (\QQ[t, s], 0)\leftarrow (\QQ[t', v, w], 0)),
 \]
 where
$$ 
\begin{array}[t]{lcr}
u \longmapsto &t &\longmapsfrom  t'\\
s' \longmapsto &s& \\
&-3s^2&\longmapsfrom v\\
&-2s^3&\longmapsfrom w
\end{array}
$$

Since the morphism $(\QQ[u, s'], 0)\to (\QQ[t, s], 0)$ is both surjective and injective, then   \linebreak $D((\QQ[u, s'], 0)\to (\QQ[t, s], 0)\leftarrow (\QQ[t', v, w], 0))$ is isomorphic to $(\QQ[t', v, w], 0)$ as differential graded algebra. Therefore  they have isomorphic cohomology which implies in particular that  $H^*_G(X)$ is a polynomial algebra, generated by three elements, and hence a Cohen--Macaulay ring as desired.
\end{example}

\begin{example}[Corank]\label{E:Corank}
Consider a generalized cohomogeneity one space $X$ given by the group diagram $(T^n, 1, S^1, T^n)$ with $n\geq 2$. Then the $T^n$-equivariant cohomology of $X$ is given by
\begin{align*}
H^*_{T^n}(X)=H^*(\B T^n, \QQ) \oplus H^*(\B S^1, \QQ)=\QQ[x, y_1, \ldots, y_n]/x\cdot y_1=\ldots = x\cdot y_n=0
\end{align*}
where $x$ generates $H^*(\B S^1,\QQ)$ and the $y_i$ generate $H^*(\B T^n, \QQ)$.

We deduce that 
\begin{align*}
\Kdim H^*_{T^n}(X)=\max\{ \Kdim H^*(\B T^n, \QQ), \Kdim H^*(\B S^1, \QQ)\}=n
\end{align*}
(see \cite[Lemma 2.3(i)]{GoMa14}) whilst $\depth H^*_{T^n}(X)=1$. It remains to quickly justify the latter.

We argue that any sequence of two polynomials in 
\begin{align*}
\QQ[x, y_1, \ldots, y_n]/x\cdot y_1=\ldots = x\cdot y_n=0
\end{align*}
cannot be a regular sequence. Indeed, any element $p$ in this ring which is not a zero-divisor necessarily is of the form $a x^{k} + b_1 y_1^{l_1}+ \ldots +b_n y_n^{l_n} + I$ with $I$ the ideal generated by the $xy_i$, with $a\neq 0$, $k\geq 1$ and, without restriction, with $b_1\neq 0$, $l_1\geq 1$.  It follows that any element in $\QQ[x, y_1, \ldots, y_n]/(x\cdot y_1,\ldots, x\cdot y_n, p)$ is a zero-divisor.

This shows that for any $2\leq n=\rank K^+- \rank H$ there exists a generalized cohomogeneity one space $X$ represented by $(K^+,1, K^-, K^+)$ satisfying all the other prerequisites of Theorem \ref{T:Main_THM_A}---the surjectivity of the maps induced in cohomology is trivial---such that the $(G=K^+)$-equivariant cohomology of $X$ is \emph{not} Cohen--Macaulay. Hence the rank condition in Theorem \ref{T:Main_THM_A} is necessary and sharp.
\end{example}

\begin{example}[Surjectivity]\label{E:Surjectivity}
We give an example of a generalized cohomogeneity one space which satisfies all the prerequisites of Theorem \ref{T:Main_THM_A}, to which the second case applies, namely $\rank H< \max \{\rank K^-, \rank K^+\}$,  and which illustrates the following:

It holds that 
\begin{align}\label{E:surjj}
\im H^*(\B\iota^{-})+\im H^*(\B\iota^{+})=H^{*}(\B H, \QQ)
\end{align}
but none of the single maps $\im H^*(\B\iota^{\pm})$ is surjective.

The example is given by the group diagram
\begin{align*}
(Sp(1)\times \Sp(1)\times \Sp(1), S^1\times S^1, S^1 \times \Sp(1)\times \Sp(1), Sp(1)\times S^1\times \Sp(1))
\end{align*}
with group inclusions specified as follows: $K^\pm\subseteq G$ come with the standard blockwise inclusion, the inclusion of $H=S^1[a]\times S^1[b]$ (into $G$ and $K^\pm$) is encoded by the matrix
$(a,b,ab)$.

Its $G$-equivariant cohomology computes as the cohomology of the algebraic double mapping cylinder given by
 \[
 D((\QQ[a,x,y],0)\to (\QQ[a,b],0) \leftarrow (\QQ[z,b,y],0)),
 \]
 where
$$ 
 \begin{array}[t]{lcr}
a\longmapsto & a &\\
&b & \longmapsfrom  b\\
x\longmapsto & b^2 &\\
y\longmapsto &(a+b)^2 &\longmapsfrom y\\
& a^2 & \longmapsfrom z
\end{array}
$$

It is clear that none of the morphisms $H^*(\B \iota^\pm)$ is surjective. Let us justify that Property \eqref{E:surjj} holds.  It is indeed easy to see that any monomial of the form $a^lb^k$ lies in the sum in \eqref{E:surjj}. For this, without restriction we can assume that $l\geq k$---in the case $k\geq l$ we argue in the analog way replacing $H^*(\B \iota^-)$ by $H^*(\B\iota^+)$, etc. Hence, we have that $a^lb^k=(ab)^k\cdot a^{l-k}$. It holds that $ab\in \im H^*(\B \iota^\pm)$ whence so is $(ab)^k$. Since $a \in \im H^*(\B \iota^-)$ it follows that $a^lb^k\in \im H^*(\B \iota^-)$.
\end{example}


\section{Equivariant cohomology and curvature}\label{S:PC}
In this section we show that for cohomogeneity one Alexandrov spaces with $\curv\geq 1$, the notion of Cohen--Macaulay and equivariant formality agree in most cases.

First, let us recall that for a positively curved (effective) cohomogeneity one Alexandrov $G$-space, as in the Riemannian case, the \emph{corank} of the principal isotropy group in  $G$ is at most $2$.  More precisely, we have

\begin{thm}[Rank Lemma]\label{T: Rank_Lemma}\cite{GGZPC}
Let $X$ be a positively curved Alexandrov space with an effective, isometric action of a compact Lie group $G$. If the action is of cohomogeneity one, then the following statements hold.
\vspace{.2cm}
\begin{enumerate}
\item If $X$ is even-dimensional, then the corank of at least one of the non-principal isotropy groups is zero, and the corank of the principal isotropy group is $1$.
\vspace{.2cm}
\item If $X$ is odd-dimensional, then either the corank of at least one of the non-principal isotropy groups is $1$, and the corank of the principal isotropy is $2$, or the coranks of all isotropy groups are zero.
\end{enumerate}
\end{thm}

First  we characterize cohomogeneity one Alexandrov spaces with $\curv\geq 1$ in terms of Euler characteristics and the coranks of the isotropy groups.
Using the additivity properties of the Euler characteristic together with the double mapping cylinder decomposition, and drawing on the fact that  a homogeneous space $G/H$ with $G$ a compact Lie group has positive Euler characteristic   if and only if $\rank H=\rank G$ (see \cite{Wang1949}),
we directly deduce the subsequent proposition from the Rank Lemma \ref{T: Rank_Lemma}.

\begin{prop}\label{prop:corank_Euler}
Let $X$ be a cohomogeneity one Alexandrov space with   $\curv\geq 1$ and with group diagram $(G, H, K^{-}, K^{+})$. Then
\begin{enumerate}
\item $\chi(X)>0$ if and only if either  $X$ is even dimensional, or $X$ is odd dimensional, the coranks of all isotropy groups are zero, and the Euler characteristic of at least one of the normal spaces of directions at the singular orbits is $1$.
\item $\chi(X)=0$ if and only if $X$ is odd dimensional and one of the following applies:
\begin{enumerate}
\item The coranks of all isotropy groups are zero and the Euler characteristics of the normal spaces of directions at singular orbits are both equal to $2$, or,
\item the corank of at least one of the non-principal isotropy groups is $1$, and the corank of the principal isotropy group is $2$.
\end{enumerate}
\item $\chi(X)<0$ if and only if $X$ is odd dimensional, the coranks of all isotropy groups are zero, and  at least one of the two Euler characteristics of the normal spaces of directions at the singular orbits is at least $2$, and the Euler characteristic of the normal space of directions at the other singular orbit is at least $3$, i.e.~without restriction
\begin{align*}
\chi(K^+/H)\geq 3 \qquad\textrm{and}\qquad \chi(K^-/H)\geq 2.
\end{align*}
\end{enumerate}
\end{prop}


\begin{prop}\label{prop:CM_ef}
Let $X$ be a cohomogeneity one Alexandrov space with   $\curv\geq 1$.
Then we deduce:
\begin{enumerate}
\item If $X$ is even dimensional, then $X$ is Cohen--Macaulay if and only if it is equivariantly formal.
\item If $X$ is odd dimensional and $\chi(X)\neq0$, then $X$ is Cohen--Macaulay if and only if it is equivariantly formal.
\item If  $X$ is odd dimensional,  $\chi(X)=0$ and $X$ is  equivariantly formal then either there exists a singular normal  fiber homeomorphic to  $\mathbb{CP}^3/\mathbb{Z}_2$ or
$X$ is equivariantly homeomorphic to a smooth manifold.
\end{enumerate}
\end{prop}
\begin{proof}
From \cite[Proposition 2.9]{GoMa14} we recall that when the rank of one of $H$, $K^\pm$ equals the rank of $G$, the cohomogeneity one $G$-action is equivariantly formal if and only if it is Cohen--Macaulay.  Hence Part (1) follows from the Rank Lemma. Part (2) is a direct consequence of Proposition \ref{prop:corank_Euler}, Items (1) and (3).

As for Part (3) we again use \cite[Proposition 2.9]{GoMa14} to see that Proposition \ref{prop:corank_Euler}, Item (2a) applies. First assume  that without restriction $K^+/H$ is odd-dimensional. Hence,  its Euler characteristic vanishes, contradicting  Item (2a) of Proposition \ref{prop:corank_Euler}.
Therefore, both $K^\pm/H$ are even dimensional, and, due to Synge, their fundamental groups are in  $\{0,\mathbb{Z}_2\}$. If both are simply-connected, they are spheres, as their Euler characteristics are $2$  by Item (2a) of Proposition \ref{prop:corank_Euler}.  Hence $X$ is a smooth manifold. If this is not the case,  then due to \cite{WZ}, they are universally covered by flag manifolds, i.e.~spaces of Euler characterstics $\chi(W^6)=\chi(W^{12})=\chi(W^{24})=6$, or by $\mathbb{CP}^{2n+1}$ (for $n\geq 0$) respectively by even dimensional spheres. By the multiplicativity of the Euler characteristic in coverings of finite CW-complexes, we deduce that the only space other than (simply-connected) spheres, which may appear as a normal fibre with Euler characteristic two is $\mathbb{CP}^3/\mathbb{Z}_2$. (In case fibres are spheres, the space is a manifold.)
\end{proof}

\begin{example}
An odd-dimensional cohomogeneity one Alexandrov space with vanishing Euler characteristic and equivariantly formal $G$-action, for example, is given by the suspension $\Susp(\mathbb{CP}^3/\mathbb{Z}_2)$ of  $\mathbb{CP}^3/\mathbb{Z}_2$ with group diagram $(Sp(2),Sp(1)U(1)\cdot \mathbb{Z}_2, Sp(2), Sp(2))$.
\end{example}


\pagebreak

\

\vfill

\begin{center}
\noindent
\begin{minipage}{\linewidth}
\small \noindent \textsc
{Manuel Amann} \\
\textsc{Institut f\"ur Mathematik}\\
\textsc{Differentialgeometrie}\\
\textsc{Universit\"at Augsburg}\\
\textsc{Universit\"atsstra\ss{}e 14 }\\
\textsc{86159 Augsburg}\\
\textsc{Germany}\\
[1ex]
\textsf{manuel.amann@math.uni-augsburg.de}\\
\textsf{www.math.uni-augsburg.de/prof/diff/arbeitsgruppe/amann/}
\end{minipage}
\end{center}


\vspace{5mm}

\begin{center}
\noindent
\begin{minipage}{\linewidth}
\small \noindent \textsc
{Masoumeh Zarei} \\
\textsc{Institut f\"ur Mathematik}\\
\textsc{Differentialgeometrie}\\
\textsc{Universit\"at Augsburg}\\
\textsc{Universit\"atsstra\ss{}e 14 }\\
\textsc{86159 Augsburg}\\
\textsc{Germany}\\
[1ex]
\textsf{masoumeh.zarei@math.uni-augsburg.de}\\
\textsf{www.math.uni-augsburg.de/prof/diff/arbeitsgruppe/zarei/}
\end{minipage}
\end{center}

\end{document}